\documentclass[11pt]{article}
\usepackage{latexsym}
\usepackage[all]{xy}

\usepackage{amsfonts}
\usepackage{amsthm}
\usepackage{amssymb}
\usepackage{mathcomp}
\usepackage{mathrsfs}
\usepackage[centertags]{amsmath}
\usepackage{calligra}
\usepackage{eucal}

\usepackage{csquotes}
\usepackage[applemac]{inputenc}

\def\dar[#1]{\ar@<2pt>[#1]\ar@<-2pt>[#1]}
  \entrymodifiers={!!<0pt,0.7ex>+}

\newcommand{\Map}{\mathrm{Map}}
\newcommand{\Hom}{\mathrm{Hom}}
\newcommand{\Tor}{\mathrm{Tor}}
\DeclareMathOperator{\hofib}{\mathrm{hofib}}
\DeclareMathOperator{\hocofib}{\mathrm{hocofib}}
\newcommand{\Mod}{\mathbf{Mod}}

\newcommand{\N}{\mathbb N}

\newcommand  {\salg}   {\mathbf{sAlg}_{k} }
\newcommand{\sset}{\mathbf{sSet}}

\theoremstyle{plain}
\newtheorem{thm}{Theorem}[section]
\newtheorem{prop}[thm]{Proposition}
\newtheorem{lem}[thm]{Lemma}

\newtheorem{cor}[thm]{Corollary}

\theoremstyle{definition}
\newtheorem{df}[thm]{Definition}
\newtheorem{rmk}[thm]{Remark}

\begin{document}

\title{\textbf{Infinitesimal and square-zero extensions of simplicial algebras} \\  \bigskip - Notes for students -}
\bigskip
\bigskip
\bigskip
\bigskip

\author{\textbf{Mauro Porta}\\
\small{Université Paris Diderot} \bigskip \\ 
\textbf{Gabriele Vezzosi}\\
\small{Institut de Mathématiques de Jussieu}}

\date{October 2013}

\maketitle

\tableofcontents

\section*{Introduction}

These notes were written to meet the requests of some students who pointed out that the exposition of the role of the cotangent complex in the Postnikov towers for simplicial commutative algebras in \cite{HAG-II} was too terse and needed some kind of unzipping.\\
We took also the opportunity to enlarge a little bit the context, by introducing square-zero extensions and their relation with infinitesimal extensions (i.e. those coming from derivations). The idea is that infinitesimal extensions are captured by the cotangent complex, that square-zero extensions are special infinitesimal extensions, and that the Postnikov tower of a simplicial commutative algebra is built out of square-zero extensions. We conclude the notes with two applications: we give connectivity estimates for the cotangent complex and we show how obstructions can be seen as deformations over simplicial rings.\\
All the material is well-known to experts but details might be useful to people meeting these topics for the first time. A similar path, in a broader and less elementary context, might be found in \cite[\S  8.3, \S8.4]{Lurie_Higher_algebra}.\\

\noindent \textbf{Notational remarks.} We denote by $\salg$ the model category of simplicial commutative $k$-algebras. All tensor products, unless differently stated, are implicitly derived.

\section{Infinitesimal extensions}

Infinitesimal extensions are defined by derived derivations:

\begin{df}\label{inf}
Let $A\to B$ be a cofibrant $A$-algebra, $M$ be a simplicial $B$-module and $\overline{d} \in \pi_0 (\mathrm{Map}_{A/\salg/B} (B, B\oplus M[1]))$ be a derived derivation from $B$ to $M[1]$, represented by a map $d \colon B \to B\oplus M[1]$ in $A / \salg/B$. If we denote by $\varphi_{d}: \mathbb{L}_{B/A}\to M[1]$ the map of $B$-modules corresponding to $d$, the \emph{infinitesimal extension} $\psi_{d}: B\oplus_{d}M \to B$ \emph{of $B$ by $M$ along $d$} is the map in $\mathrm{Ho}(A/\salg/B)$ defined by the following homotopy cartesian diagram in $A/\salg$
\[
\xymatrix{B\oplus_{d}M \ar[d]_-{\psi_{d}} \ar[r] & B \ar[d]^-{0} \\B \ar[r]_-{d} & B\oplus M[1] }
\]
where $0$ denotes the section corresponding to the trivial derived derivation $0 \colon \mathbb L_{B/A} \to M[1]$.
\end{df}

\noindent The appearance of $M$ (instead of any shift of it) in the notation $ B\oplus_{d}M$ calls for an explanation.

\begin{prop} \label{prop infinitesimal fiber}
If $\psi_{d}: B\oplus_{d}M \to B$ is an infinitesimal extension of $B$ by $M$ along $d$, then the homotopy fiber of $\psi_{d}$ at $0$ is isomorphic to $M$ in $\mathrm{Ho}(B\textrm{-}\Mod)$. 
\end{prop}

\begin{proof}
Proposition \ref{prop homotopy fibre} shows that
\[
\hofib \psi_{d} \simeq \hofib (0 \colon B \to B \oplus M[1])
\]
(where the fibres are taken over $0$). In order to explicitly compute $\hofib(0)$ we observe that the square
\[
\xymatrix{
B \ar[r]^-0 \ar[d] & B \oplus M[1]\ar[d]^p \\ 0 \ar[r]& M[1]
}
\]
is homotopy cartesian: in fact, $p$ is a fibration (being surjective) and every object is fibrant, so that the statement follows from Corollary \ref{cor computing homotopy pullback} and from the fact that the previous diagram is obviously a strict pullback. As consequence, the outer rectangle in
\[
\xymatrix{
\hofib 0 \ar[r] \ar[d] & 0 \ar[d] \\ B \ar[d] \ar[r]^-0 & B \oplus M[1] \ar[d] \\ 0 \ar[r] & M[1]
}
\]
is a homotopy pullback, so that
\[
\hofib \psi_d \simeq \hofib 0 \simeq \Omega(M[1])
\]
Now, $M[1] = M \otimes_A A[S^1]$ is the suspension of $M$ and $\Omega \Sigma(M) \simeq M$ by Corollary \ref{cor omega sigma}.
\end{proof}

\section{Square-zero extensions}

\begin{df} \label{sqzero}
Let $n\geq 0$. Let $A\in \salg$, $B_1$ a cofibrant $A$-algebra, and $I\subseteq \pi_{n}(B_1)$ a sub-$\pi_0(B_1)$-module. A morphism of cofibrant $A$-algebras $\varphi \colon B_1 \to B_0$ in $A/\salg$ is a \emph{$A$-square-zero extension by $I[n]$} if
the following conditions are met
\begin{enumerate}
\item $B_{0}$ and $B_1$ are $n$-truncated;
\item $\varphi$ is an $(n-1)$-equivalence of $A$-algebras;
\item For any $n$-truncated $A$-algebra $E$, the following diagram is homotopy cartesian
\[
\xymatrix{
\Map_{A/\salg}(B_{0},E) \ar[d] \ar[rr]^-{\Map(\varphi, E)} & & \Map_{A/\salg}(B_{1},E) \ar[d] \\
[B_{1},E]_{0, I} \ar[rr] & & [B_1, E]
}
\]
where $[B_1, E]$ denotes the set of homotopy classes of maps $B_1 \to E$, and $[B_{1},E]_{0,I}$ the subset of $[B_1,E]$ consisting of those $[f]$ such that $\pi_n (f)$ is zero on $I$;
\item The canonical map $\pi_{n}(B_{1}) \rightarrow \pi_{n}(B_0)$ is surjective with kernel $I$, i.e. $\pi_{n}(B_1)/I \simeq \pi_n(B_0)$;
\item if $n=0$ then $I^2=0$ (classical case).
\end{enumerate}
\end{df}

\begin{rmk} \label{rmk square zero extension}
\begin{enumerate}
\item Equivalently, we can define $[B_1, E]_{0,I}$ as the following (homotopy) pullback in $\sset$:
\[
\xymatrix{
[B_{1},E]_{0, I} \ar[rr] \ar[d] & & [B_{1},E] \ar[d] \\ \Hom_{\pi_{0}(A)\textrm{-}\Mod}(\pi_n(B_1)/I, \pi_n(E)) \ar[rr]_-{\textrm{can}} & & \Hom_{\pi_{0}(A)\textrm{-}\Mod}(\pi_n(B_1), \pi_n(E))
}
\]
In fact, inspection reveals that the above diagram is a strict pullback. It is a homotopy pullback because every object there is discrete, hence fibrant and, as consequence, the maps are fibrations.

\item For $n=0$, and $A=k$ we get back the classical definition of square-zero extension of commutative $k$-algebras.

\item If $B_1 \rightarrow B_0$ is an $A-$square-zero extension by $I[n]$, then $I$ is canonically a $\pi_{0}(B_0)$-module. This follows from $\pi_{0}(B_0)\simeq \pi_{0}(B_1)$, if $n>0$, and is classical if $n=0$ (since $I^2=0$).
\end{enumerate}
\end{rmk}

\begin{lem} \label{lemma eilenberg maclane}
If $\varphi \colon B_1 \to B_0$ is a square-zero extension by $I[n]$ in $A/\salg$, then $\hofib \varphi$ is a $K(I,n)$-space.
\end{lem}

\begin{proof}
We have by definition a fibre sequence
\[
\hofib \varphi \to B_1 \xrightarrow{\varphi} B_0
\]
in $A\textrm{-}\Mod$ (and therefore a fibre sequence of pointed simplicial sets). The long exact sequence of homotopy groups shows then that
\[
\pi_m(\hofib \varphi) = 0
\]
if $m > n$ or $m < n-1$. Moreover, for $m = n - 1$ we have
\[
\pi_n(B_1) \twoheadrightarrow \pi_n(B_0) \to \pi_{n-1}(\hofib \varphi) \to \pi_{n-1}(B_1) \stackrel{\sim}{\to} \pi_{n-1}(B_0)
\]
so that $\pi_{n-1}(\hofib \varphi) = 0$. Finally, we have a short exact sequence
\[
0 \to \pi_n(\hofib \varphi) \to \pi_n(B_1) \to \pi_n(B_0) \to 0
\]
so that axiom 4. readily implies that
\[
\pi_n(\hofib \varphi) \simeq I
\]
completing the proof.
\end{proof}

\begin{prop} \label{unique}
Let $n\geq 0$, $A\in \salg$, $\varphi:B_1 \to B_0$ and $\varphi':B_1 \to B_0'$ in $A/\salg$ two $A$-square-zero extensions by $I[n]$ ($I\subseteq \pi_{n}(B_1)$ a fixed sub-$\pi_0(B_1)$-module). Then there is an isomorphism $B_0 \simeq B_0'$ in $\mathrm{Ho}(B_1/\salg)$.
\end{prop}

\begin{proof}
The mapping space axiom 3. tells us that the simplicial sets $\Map_{A/\salg}(B_0, E)$ and $\Map_{A/\salg}(B_0', E)$ are isomorphic in $\mathrm{Ho}(\mathbf{sSet})$, for any $n$-truncated $E\in A/\salg$. In particular, by taking $E= B_0'$, we get a map $u \colon B_0 \rightarrow B_0'$. Denote as $(A/\salg)_{\leq n}$ the left Bousfield localization of $A/\salg$ with respect to the single map $S:={S^{n+1}\otimes A[T]\to A[T]}$, and denote the left Quillen adjoint by $\tau_{\leq n}:A/\salg \rightarrow (A/\salg)_{\leq n}$. The fibrant objects in $(A/\salg)_{\leq n}$ are the $S$-local objects, i.e. $n$-truncated simplicial $A$-algebras. The homotopy category of $(A/\salg)_{\leq n}$ is identified as the full subcategory of $\mathrm{Ho}(A/\salg)$ consisting of $n$-truncated objects.
Now, the mapping space axiom (3) implies that, for any $S$-local object $E\in \salg$, the map
\[
u^*: \Map_{A/\salg}(B_0, E) \to \Map_{A/\salg}(B_0', E)
\]
is an isomorphism in $\mathrm{Ho}(\mathbf{sSet})$, i.e. ( \cite[Proposition 3.5.3]{Hirschhorn_Model_2003}) $u \colon B_0 \rightarrow B_0'$ is an $S$-local equivalence. But both $B_0$ and $B_0'$ are $S$-local objects (being $n$-truncated), so we conclude that in fact $u$ is a weak equivalence in $A/\salg$ (an $S$-local equivalence between $S$-local objects is a weak equivalence: $S$-local Whitehead Theorem Theorem 3.2.13 loc.\ cit.). How do we climb up to an equivalence of $B_{1}/\salg$? Simply observe that the isomorphism $\Map_{A/\salg}(B_0, E) \simeq \Map_{A/\salg}(B_0', E)$ (in $\mathrm{Ho}(\mathbf{sSet})$), from which we deduced the map $u$, in fact commutes (up to homotopy) with the maps
\[
\xymatrix{
\Map_{A/\salg}(B_{0},E)  \ar[rr]^-{\Map(\varphi, E)} & & \Map_{A/\salg}(B_{1},E)
}
\]
\[
\xymatrix{
\Map_{A/\salg}(B_{0}',E)  \ar[rr]^-{\Map(\varphi', E)} & & \Map_{A/\salg}(B_{1},E)
}
\]
Therefore, we may choose $u \colon B_0 \rightarrow B_0' $ as a map in $\mathrm{Ho}(B_1/\salg)$.
\end{proof}

Let $B_1 \to B_0$ be a square-zero extension by $I[n]$. We saw in Lemma \ref{lemma eilenberg maclane} and in Proposition \ref{unique} that the sub-$\pi_0(B_0)$-module controls every information about the extension; in particular, the homotopy fiber is determined and there are no two different square-zero extensions associated to the same sub-$\pi_0(B_0)$-module. We are going now to show that every sub-$\pi_0(B_0)$-module determines a square-zero extension:

\begin{prop}\label{existence}
Let $n \geq 0$. Given a cofibrant and $n$-truncated $B_1 \in A/\salg$, and a sub-$\pi_0(B_1)$-module $I\subseteq \pi_{n}(B_1)$ (such that $I^2=0$ if $n=0$), there exists a square zero extension $B_1 \to B_0$ by $I[n]$. Moreover any other such extension $B_1 \to B_{0}'$ is isomorphic to $B_{1}\to B_{0}$ in $\mathrm{Ho}(B_{1}/\salg)$.
\end{prop}

\begin{proof}
The uniqueness statement is Proposition \ref{unique}, so that we are left to prove the existence.

\noindent The idea of the proof is to construct $B_0$ as ``$B_1 /I$'' (i.e. to kill $I$ inside $B_1$) and then to take the $n$-truncation as an $A$-algebra. To begin with, let us consider $I$ as an $A$-module (via $A\to \pi_0(A) \to \pi_0 (B_1)$); the category $A \textrm{-} \Mod$ being monoidal model we have a canonical identification
\[
\Hom_{\mathrm{Ho}(A\textrm{-}\Mod)}(I, \pi_n(B_1)) \simeq \Hom_{\mathrm{Ho}(A \textrm{-} \Mod)}(I[n], B_1)
\]
so that the inclusion $I \subseteq \pi_n(B_1)$ induces a map of $A$-modules $I[n]\to B_1$ (because in $A \textrm{-} \Mod$ every object is fibrant, hence maps in the homotopy category can be represented in $A \textrm{-} \Mod$). At this point, we obtain by adjuntion an induced map of $A$-algebras
\[
\mathrm{Sym}_{A}(I[n]) \to B_1
\]
Define a new object $\widetilde{B_0}$ via the following pushout square in $A/\salg$:
\[
\xymatrix{
\mathrm{Sym}_{A}(I[n])  \ar[d] \ar[r]^-{0} & A \ar[d] \\
B_1 \ar[r] & \widetilde{B_0}
}
\]
where the map $0$ is induced by the zero map of $A$-modules $I[n] \to A$. Finally, introduce $B_0 := \tau_{\leq n} \widetilde{B_0}$. $B_0$ comes equipped with a canonical map
\[
\varphi \colon B_1 \to \widetilde{B_0} \to \tau_{\leq n} \widetilde{B_0}=B_0
\]
We claim that $\varphi$ is the square-zero extension by $I[n]$ we were looking for. Let us check that the conditions of Definition \ref{sqzero}:
\begin{enumerate}
\item $B_1$ is $n$-truncated by hypothesis, while $B_0$ is $n$-truncated by construction;
\item in order to show that $\varphi$ is an $(n-1)$-equivalence of $A$-algebras, and that the canonical map $\pi_{n}(B_{0}) \rightarrow \pi_{n}(B_1)$ induces an isomorphism $\pi_{n}(B_1)/I \simeq \pi_n(B_0)$, we use the spectral sequence of \cite[Theorem II.6.b]{Quillen_Homotopical_1967}. Set first of all $R_* :=\pi_* (\mathrm{Sym}_{A}(I[n]))$, so that the spectral sequence reads off as:
\[
E_{pq}^2= \Tor_{p}^{R_{*}}(\pi_*B_1,\pi_* A)_q \Rightarrow \pi_{p+q}(\widetilde{B_0})
\]
Let $C_{\bullet *}\to \pi_{*}B_1$ be a flat resolution of $\pi_*B_1$ as a graded$_{*}$ $R_*$-module, so that
\[
\textrm{Tor}_{p}^{R_{*}}(\pi_*B_1,\pi_* A)_q = H^p((C_{\bullet *}\otimes_{R_*} \pi_* A)_{q})
\]
Let us compute the degree $q$ part of $C_{\bullet *} \otimes_{R_*} \pi_* A$:
\[
(C_{\bullet *}\otimes_{R_*} \pi_* A)_{q} = \{ x_{ij}\otimes y_k \mid x_{ij}\in C_{ij}, y_k\in \pi_k A, \, j+k=q\}
\]
for $q\leq n$.
\begin{itemize}
\item If $q < n$, then $k < n$ and there are elements $\widetilde{y}_k \in R_k \simeq \pi_k A$ mapping to $y_k$, so that
\[
\{ x_{ij}\otimes y_k \mid x_{ij}\in C_{ij}, \: y_k\in \pi_k A, \: j+k=q \} = \{\widetilde{y}_k x_{ij}\otimes 1 \mid x_{ij}\in C_{ij}, \: \tilde{y_k}\in R_k , \: j+k=q\}
\]
and therefore
\[
(C_{\bullet *}\otimes_{R_*} \pi_* A)_{q}\simeq C_{\bullet q}, \quad \textrm{for} \: 0\leq q<n
\]

\item If $q=n$, for $j > 0$ (hence $k<n$) we still have 
$$ x_{ij}\otimes y_k = \tilde{y}_{k} x_{i,0}\otimes 1 $$
while for $j=0$, since $R_n\simeq \pi_n (A)\oplus I$, we get instead $$ x_{i0}\otimes y_n = (y_n,0)\cdot x_{i,0}\otimes 1 = (y_n,\xi)\cdot x_{i0}\otimes 1$$ for any $\xi \in I$ (and $x_{i0}\in C_{i0}$, $y_n\in \pi_{n}A$). Therefore $$(C_{\bullet *}\otimes_{R_*} \pi_* A)_{n}\simeq C_{\bullet n}/I\cdot C_{\bullet 0},$$ where $I\cdot C_{\bullet 0}:=\{ (0,\xi)\cdot x_{\bullet 0} \, | \, \xi \in I \subset R_n,\, x_{\bullet 0}\in C_{\bullet 0}\}$.
\end{itemize}
Therefore
\[
\textrm{Tor}_{p}^{R_{*}}(\pi_*B_1,\pi_* A)_q = \begin{cases}
H^p(C_{\bullet q}) = \delta_{p0}\cdot \pi_{q}B_1 & \text{if } 0 \le q < n \\
H^p(C_{\bullet n}/I\cdot C_{\bullet 0}) & \text{if } q = n
\end{cases}
\]
Let us compute $H^0(C_{\bullet n} / I \cdot C_{\bullet 0})$. Introduce first of all the ideal
\[
J := I \oplus \bigoplus_{q > n} R_q
\]
so that, given any graded $R_*$-module $M_*$ we have
\[
M_n / I \cdot M_0 \simeq (M_* / J \cdot M_*)_n \simeq (M_* \otimes_{R_*} R_* / J)_n
\]
and now observe that we are given an exact sequence
\[
C_{1,*} \to C_{0,*} \to \pi_*(B_1) \to 0
\]
Tensoring with $R_* / J$ preserves the right exactness, and taking the degree $n$ part is obviously an exact functor, so that we obtain an exact sequence
\[
C_{1,n} / I \cdot C_{1,0} \to C_{0,n} / I \cdot C_{0,0} \to \pi_n(C) / I \to 0
\]
which readily implies that
\[
H^0(C_{\bullet n}/I\cdot C_{\bullet 0})\simeq \pi_{n}(B_1)/I
\]
so the $E^2$ page of our homological spectral sequence is first quadrant and we can draw it as properties 
%\begin{itemize}
%\item it is first quadrant and $d^2: E^2_{p,q} \to E^2_{p+2, q-1}$;
%\item the horizontal semi-strip $0\leq q <n$ has only a priori nonzero entries for $p=0$ and those are $\pi_p B_1$;
%\item the $q=n$ semi-line starts with $\pi_n (B_1)/I$ (for $p=0$)
%\end{itemize}
%i.e. we can draw it as
\[
\xymatrix{
& & & & & &   \\
q= n & \pi_n(B_1)/I \ar[rrd] & \bullet & \bullet & \bullet & \bullet  & \bullet    \\ q=n-1 & \pi_{n-1}(B_1)/I \ar[rrd] & 0 & 0 & 0 & 0 & 0  \\q= n-2 & \pi_{n-2}(B_1)/I \ar[rrd]& 0 & 0 & 0 & 0 & 0   \\ \vdots & & & & & &   \\ \vdots & & & & & &  \\    q= 0 & \pi_0(B_1)/I & 0 & 0 & 0 & 0 & 0    \\ & p=0 & p=1 & p=2 & p=3 & p=4 & \ldots
}
\]
Thus $E^{\infty}_{pq}= E_{pq}^2$ for $0\leq p+q \leq n$, so that
\[
\pi_{i}(\widetilde{B_1}) = \begin{cases}
\pi_i(B_1) & \text{if } 0 \le i < n \\
\pi_n(B_1) / I & \text{if } i = n
\end{cases}
\]

\item For any $A$-algebra $E$, the following diagram consists of homotopy cartesian squares 
\[
\xymatrix{
\Map_{A/\salg}(\widetilde{B_{0}},E) \ar[d] \ar[rr]^-{\Map(\varphi, E)} & & \Map_{A/\salg}(B_{1},E) \ar[d] \\
\Map_{A/\salg}(A,E) \ar[rr] \ar[d] & & \Map_{A/\salg}(\mathrm{Sym}_{A}(I[n]),E) \ar[d] \\
[A,E] \simeq [\tau_{\leq n} A \oplus I[n],E]_{\textrm{0, I[n]}} \ar[rr] & & [\tau_{\leq n} A \oplus I[n],E] \simeq [A\oplus I[n], E]
}
\]
(for the top square we use the homotopy pushout definition of $\widetilde{B_0}$ and the fact that $\Map_{A/\salg}(\widetilde{B_0}, E) \simeq \Map_{A/\salg}(B_0:=\tau_{\leq n}\widetilde{B_0}, E)$ since $E$ is $n$-truncated; for the bottom square we use that $\tau_{\leq n}(\mathrm{Sym}_{A}(I[n]))\simeq \tau_{\leq n}(A\oplus I[n]) \simeq \tau_{\leq n}(A)\oplus I[n]$, and again the hypothesis that $E$ is $n$-truncated). To conclude it just remains to remark that the diagram of sets
\[
\xymatrix{
[B_1,E]_{0, I} \ar[d] \ar[r] & [B_1, E] \ar[d] \\
[A,E] \ar[r] & [A\oplus I[n], E]
}
\]
is cartesian.
\end{enumerate}
\end{proof}

\noindent The following result will be useful later

\begin{lem}\label{lem} Let $n\geq 0$, and $\varphi:B_1 \to B_0$ in $A/\salg$ a square-zero extension by $I[n]$ ($I\subseteq \pi_{n}(B_1)$ a fixed sub-$\pi_0(B_1)$-module). Let $\widetilde{B_{1}}$ be defined by the following  pushout square in $A/\salg$ $$\xymatrix{\mathrm{Sym}_{A}(I[n])  \ar[d] \ar[r]^-{0} & A \ar[d] \\ B_1 \ar[r] & \widetilde{B_0}  }$$ where the map $0$ is induced by the zero map of ($\pi_{0}A$ hence) $A$-modules $I[n]\to A$. Then 
\begin{itemize}
\item there is a canonical isomorphism $B_0 \simeq \tau_{\leq n}\widetilde{B_0}$  in $\mathrm{Ho}(B_{1}/\salg)$;
\item there is a canonical isomorphism $\widetilde{B_0}\otimes_{B_1}B_0 \simeq \mathrm{Sym}_{B_0} I[n+1]$ in $\mathrm{Ho}(B_0/\salg)$.
\end{itemize}
\end{lem}

\begin{proof}
The proof of the first assert is part of the proof of Proposition \ref{existence}. Let us prove the second part of the statement. We have the following ladder of homotopy pushouts in $\mathrm{Ho}(A/\salg)$:
\[
\xymatrix{
\mathrm{Sym}_{A}(I[n]) \ar[d] \ar[r]^-{0} & A \ar[d] \\
B_1 \ar[d] \ar[r] & \widetilde{B_0} \ar[dd] \\
\widetilde{B_0} \ar[d] & \\
B_0 \ar[r] &  \widetilde{B_0}\otimes_{B_1}B_0
}
\]
Now, by the upper homotopy cocartesian square, the composite $\mathrm{Sym}_{A}(I[n]) \to B_1 \to \widetilde{B_0}$ is isomorphic (in $\mathrm{Ho}(A/\salg)$) to the composite $\xymatrix{\mathrm{Sym}_{A}(I[n]) \ar[r]^-{0} & A \ar[r] & \widetilde{B_0}}$, so that the following square
\[
\xymatrix{
\mathrm{Sym}_{A}(I[n]) \ar[d]_-{0} \ar[r]^-{0} & A \ar[dd] \\
A \ar[d]  & \\
B_0 \ar[r] &  \widetilde{B_0}\otimes_{B_1}B_0
}
\]
is homotopy cocartesian as well. Therefore, if  $C$ is defined by the homotopy pushout
\[
\xymatrix{
\mathrm{Sym}_{A}(I[n]) \ar[d]_-{0} \ar[r]^-{0} & A \ar[d] \\
A \ar[r] & C
}
\]
there is an induced homotopy pushout
\[
\xymatrix{
A \ar[r] \ar[d] & C \ar[d] \\
B_0 \ar[r] &  \widetilde{B_0}\otimes_{B_1}B_0
}
\]
Let us compute $C$. In order to do this, we observe that $\mathrm{Sym}_{A}: A \textrm{-} \Mod \to A/\salg$ is left Quillen, hence it commutes with homotopy pushouts; since $A\simeq \mathrm{Sym}_{A}(0)$, we get that $C\simeq \mathrm{Sym}_{A}(P)$, where $P$ is defined by the homotopy pushout (in $A \textrm{-} \Mod$)
\[
\xymatrix{
I[n] \ar[r] \ar[d] & 0 \ar[d] \\
0 \ar[r] & P
}
\]
But, by definition of suspension functor in $A \textrm{-} \Mod$, we have then that $P\simeq \Sigma I[n]= I[n+1]$. Therefore $C \simeq \mathrm{Sym}_{A}I[n+1]$.

Now, coming back to the homotopy pushout
\[
\xymatrix{
A \ar[r] \ar[d] & C \ar[d] \\
B_0 \ar[r] &  \widetilde{B_0}\otimes_{B_1}B_0
}
\]
and recalling the base change property of the functor $\mathrm{Sym}_{-}$, we conclude that there is a canonical isomorphism $\widetilde{B_0}\otimes_{B_1}B_0 \simeq \mathrm{Sym}_{B_0}I[n+1]$ in  $\mathrm{Ho}(A/\salg)$. By tracing back the construction of this isomorphism, we see that it is indeed an isomorphism in $\mathrm{Ho}(B_0/\salg)$ (since the $B_0$-algebra structure comes in both cases from the bottom horizontal map of the pushout diagrams).
\end{proof}

\section{Any square-zero extension is an infinitesimal extension}

\begin{thm}\label{sqzerotoinf} Let $n\geq 0$, $A\in \salg$, and $u:B_1 \to B_0$ in $A/\salg$ a square-zero extension by $I[n]$ ($I\subseteq \pi_{n}(B_1)$ a fixed sub-$\pi_0(B_1)$-module). Then there exists a derived $A$-derivation $d_{u}$ of $B_0$ into $I[n+1]$, and an isomorphism $B_0\oplus_{d_u} I[n] \simeq B_1$ in $\mathrm{Ho}(A/\salg /B_0)$. Moreover, such a $d_u$ is uniquely determined as a map in $\mathrm{Ho}(A/\salg/B_0)$.
\end{thm}

\begin{proof}
Throughout the proof, recall our standing convention $\otimes \equiv \otimes^{\mathbb{L}}$. Consider the fiber - cofiber sequence of $A$-modules
\[
\xymatrix{
I[n] \ar[r] & B_1 \ar[r]^-u & B_0
}.
\]
It induces a fiber - cofiber sequence
\[
\xymatrix{
B_1 \ar[r]^-{u} & B_0 \ar[r] & I[n+1]
}.
\]
The idea is now to apply $(-)\otimes_{B_1}B_0$ to this sequence in order to obtain a split sequence; the one of the $B_0$-algebra structures on $B_0 \otimes_{B_1} B_0$ will induce the zero derivation while the other one will induce a derivation $d_u$ such that $B_1 \simeq B_0 \times_{B_0 \oplus_{d_u}I[n]} B_0$.  Let us work this idea out.

The sequence of $B_0$-modules 
\[
\xymatrix{
B_0 \simeq B_1 \otimes_{B_1}B_0 \ar[r]^-{u\otimes \textrm{id}} & B_0 \otimes_{B_1} B_0 \ar[r]  & B_0 \otimes_{B_1} I[n+1]
}
\]
is clearly split by the product map $B_0 \otimes_{B_1} B_0 \to B_0$; therefore we get a canonical isomorphism $$B_0\otimes_{B_1} B_0 \simeq B_0 \oplus (B_0 \otimes_{B_1} I[n+1])$$ in the homotopy category of $B_0$-modules.

\noindent Let $\widetilde{B_0}:= B_1 \otimes _{\mathrm{Sym}_{A}I[n]} A$, and let $\gamma \colon \tau_{\leq n}\widetilde{B_0} \to B_0$ be the isomorphism of $B_1$-algebras produced by Lemma \ref{lem}.  Introduce the morphism
\[
t \colon \widetilde{B_0} \to \tau_{\le n} \widetilde{B_0} \xrightarrow{\gamma} B_0
\]
and consider the induced map
\[
\theta:= \textrm{id}\otimes_{B_1} t \colon B_0 \otimes _{B_1} \widetilde{B_0} \longrightarrow B_0 \otimes_{B_1} B_0
\]
which is a map of $B_0$-algebras, if we endow $B_0 \otimes_{B_1} B_0$ with the $B_0$-algebra structure given by
\[
j_1 \colon B_0 \to B_0 \otimes_{B_1} B_0, \quad b \longmapsto b \otimes 1.
\]
We claim that $$\tau_{\leq n+1}\theta: \tau_{\leq n+1}(B_0 \otimes _{B_1} \widetilde{B_0}) \longrightarrow \tau_{\leq n+1}(B_0 \otimes_{B_1} B_0)$$ is an isomorphism in $\textrm{Ho}(B_0/\salg)$. Let us prove this claim.

$\diamondsuit$ We compute how $\tau_{\leq n+1}\theta$ acts on homotopy groups. Let us first compute $\pi_{i}(B_0 \otimes_{B_1}I[n+1])$ by using the spectral sequence (\cite[II \S 6, Thm. 6.c]{Quillen_Homotopical_1967})
\[
\pi_p(\pi_q(B_0)[0]\otimes_{\pi_0 B_1}I[n+1]) \Rightarrow \pi_{p+q}(B_0 \otimes_{B_1}I[n+1]).
\]
We have
\[
\pi_p(\pi_q(B_0)[0]\otimes_{\pi_0 B_1}I[n+1]) = \begin{cases}
\pi_q(B_0) \otimes_{\pi_0(B_1)} I & \text{if } p = n+1 \\
0 & \text{if } p \neq n + 1
\end{cases}
\]
so the spectral sequence degenerates, and we have for $q=0, p=n+1$
\[
\pi_{n+1}(B_0 \otimes_{B_1}I[n+1])=\pi_0(B_0)\otimes_{\pi_0(B_1)} I
\]
Now, if $n = 0$ both $B_0$ and $B_1$ are discrete, $B_0 \simeq B_1 / I$ as $B_1$-algebra and $I^2 = 0$, so that
\[
\pi_0(B_0) \otimes_{\pi_0(B_1)} I \simeq I / I^2 \simeq I
\]
If, instead, $n > 0$, then $\pi_0(B_1) \simeq \pi_0(B_0)$ and so
\[
\pi_0(B_0) \otimes_{\pi_0(B_1)} I \simeq I
\]
In conclusion we obtain
\[
\pi_i(B_0 \otimes_{B_1}I[n+1]) = \begin{cases}
0 & \text{if } i < n+1 \\
I & \text{if } i = n+1 \\
\pi_q(B) \otimes I & \text{if } i = n + 1 +q, \: q > 0
\end{cases}
\]
Since
\[
B_0\otimes_{B_1} B_0 \simeq B_0 \oplus (B_0 \otimes_{B_1} I[n+1]),
\]
we conclude that
\[
\pi_i(B_0\otimes_{B_1} B_0) = \begin{cases}
\pi_i(B_0) & \text{if } i < n+1 \\
\pi_{n+1}(B_0) \oplus I & \text{if } i = n + 1 \\
\pi_i(B_0) \oplus (\pi_q(B) \otimes I) & \text{if } i = n + 1 + q, \: q > 0
\end{cases}
\]
On the other hand, by Lemma \ref{lem}, $$B_0 \otimes _{B_1} \widetilde{B_0} \simeq \mathrm{Sym}_{B_0} I[n+1] = B_0 \oplus I[n+1] \oplus R$$ where $R$ is $(n+1)$-connected (i.e. its $\pi_i$'s vanish for $i\leq n+1$), so that there is an isomorphism $$\tau_{\leq n+1} (B_0 \otimes _{B_1} \widetilde{B_0}) \simeq B_0\oplus I[n+1]$$ in the homotopy category of $B_0$-algebras.

The reader may check that the follwoing diagram is commutative
%\texttt{WE ARE LEFT TO CHECK THAT}
\[
\xymatrix{
B_0 \otimes_{B_1} \widetilde{B_0} \ar[rr]^-\theta \ar[dr] & & B_0 \otimes_{B_1} B_0 \ar[dl] \\
& B_0 \oplus I[n+1]
}
\]
%\texttt{COMMUTES.} 
This concludes our proof of the claim that $\tau_{\leq n+1}\theta $ is an equivalence. \,\,\,\,
$\diamondsuit$\\

So we have proved that
\[
\theta_{\leq n+1}:=\tau_{\leq n+1}\theta \colon \tau_{\leq n+1}(B_0 \otimes _{B_1} \widetilde{B_0})\simeq B_0 \oplus I[n+1] \longrightarrow  \tau_{\leq n+1}(B_0 \otimes_{B_1} B_0)
\]
is an isomorphism in $\textrm{Ho}(B_0/\salg)$, and note that the $B_0$-algebra structure on the lhs is given by the map $\varphi_0$ corresponding to the zero derivation. Now we can use the other $B_0$-algebra structure
\[
j_2 \colon B_0 \to B_0 \otimes_{B_1} B_0, \quad b \longmapsto 1\otimes b,
\]
to produce the derivation we are looking for. Let us define
\[
\xymatrix{
\varphi_{d_u} \colon B_0 \simeq \tau_{\leq n+1}B_0 \ar[r]^-{\tau_{\leq n+1}j_2} & \tau_{\leq n+1}(B_0 \otimes_{B_1} B_0) \ar[r]^-{\theta_{\leq n+1}^{-1}} &  B_0 \oplus I[n+1]
}
\]
and observe that this is indeed a map in $\mathrm{Ho}(A/\salg/B)$, so it does correspond to a derived derivation $d_u: B_0 \to I[n+1]$ over $A$. Consider the corresponding infinitesimal extension defined by the homotopy pushout
\[
\xymatrix{
B_0\oplus_{d_u}I[n] \ar[d]_-{\psi_{d_u}} \ar[r]^-{\psi'} & B_0 \ar[d]^-{\varphi_{0}} \\
B_0 \ar[r]_-{\varphi_{d_u}} & B_0\oplus I[n+1]
}
\]
and observe that, since the diagram
\[
\xymatrix{
B_1 \ar[r]^-{u} & B_0 \ar@<.4ex>[r]^-{j_1} \ar@<-.4ex>[r]_-{j_2} & B_0\otimes_{B_1} B_0
}
\]
equalizes, the same is true for the diagram
\[
\xymatrix{
B_1 \ar[r]^-{u} & B_0 \ar@<.4ex>[r]^-{j_1} \ar@<-.4ex>[r]_-{j_2} & B_0\otimes_{B_1} B_0 \ar[r] & \tau_{\leq n}(B_0\otimes_{B_1} B_0) \simeq B_0\oplus I[n+1]
}
\]
and therefore, by definition of $\varphi_{0}$ (induced by $j_1$) and $\varphi_{d_u}$ (induced by $j_2$), the same is true for the diagram 
\[
\xymatrix{
B_1 \ar[r]^-{u} & B_0 \ar@<.4ex>[r]^-{\varphi_0} \ar@<-.4ex>[r]_-{\varphi_{d_u}} & B_0\oplus I[n+1]
}.
\]
So, we have an induced map
\[
\alpha: B_1 \to B_0\oplus_{d_u} I[n]
\]
in $\textrm{Ho}(A/\salg/B_0)$ (where $B_0\oplus_{d_u}I[n]$ is considered as an algebra over $B_0$ via the map $\psi_{d_u}$).

We are left to prove that $\alpha$ is an isomorphism. In order to do this, we will show that, in the following commutative diagram whose lines are fiber sequences, the map $\beta$ is a weak equivalence:
\[
\xymatrix{
\mathrm{hofib}(u) \ar[r] \ar[d]^\beta & B_1 \ar[d]^\alpha \ar[r]^u & B_0 \ar@{=}[d] \\
\mathrm{hofib}(\psi_{d_u}) \ar[r] \ar[d] & B_0 \oplus_{d_u} I[n] \ar[r]^-{\psi_{d_u}} \ar[d]^{\psi'} & B_0 \ar[d]^{\varphi_{d_u}} \\
\mathrm{hofib}(\varphi_0) \ar[r]^{\varphi_0} & B_0 \ar[r] & B_0 \oplus I[n+1]
}
\]
Proposition \ref{prop homotopy fibre} implies that the morphism
\[
\mathrm{hofib}(\psi_{d_u}) \to \mathrm{hofib}(\varphi_0)
\]
is a weak equivalence. Using the 2-out-of-3 property, it is sufficient to check that the composition
\[
\mathrm{hofib}(u) \to \mathrm{hofib}(\psi_{d_u}) \to \mathrm{hofib}(\varphi_0)
\]
is a weak equivalence. The definition of $\alpha$ implies $\psi' \circ \alpha = u$; moreover $\mathrm{hofib}(u)$ and $\mathrm{hofib}(\varphi_0)$ are (separately) isomorphic to $I[n]$. As consequence, it is sufficient to show that the left square in the following diagram
\[
\xymatrix{
I[n] \ar[r]^\gamma \ar@{=}[d] & B_1 \ar[d]^u \ar[r]^u & B_0 \ar[d]^{\varphi_{d_u}} \\
I[n] \ar[r]^\delta & B_0 \ar[r]^-{\varphi_0} & B_0 \oplus I[n+1]
}
\]
commutes \emph{in the homotopy category}, where $\gamma$ and $\delta$ denote the canonical morphisms
\[
\gamma \colon I[n] \simeq \mathrm{hofib}(u) \to B_1, \qquad \delta \colon I[n] \simeq \mathrm{hofib}(\varphi_0) \to B_0
\]
Recall from Proposition \ref{prop infinitesimal fiber} that the morphism $\delta$ is obtained from the diagram
\[
\xymatrix{
I[n] \ar@/_1pc/[ddr]_0 \ar@/^1pc/[drr] \ar@{.>}[dr]^\delta \\
& B_0 \ar[r] \ar[d]_{\varphi_0} & 0 \ar[d] \\
& B_0 \oplus I[n+1] \ar[r] & I[n+1]
}
\]
so that $\varphi_0 \circ \delta \simeq 0 \simeq \varphi_0 \circ u \circ \gamma$. Since $\varphi_0$ is a section of the canonical projection $B_0 \oplus I[n+1] \to B_0$, it is in particular a split mono; as consequence, its image in the homotopy category is a (split) mono as well. We therefore get $\delta \simeq 0 \simeq u \circ \gamma$, completing the proof.
\end{proof}

\section{Application to Postnikov towers}

\begin{prop} \label{prop postnikov sqzero}
Let $n\geq 1$, and $C\in \salg$. Then the $n$-th stage $p_n: C_{\leq n} \to C_{\leq n-1}$ of the Postnikov tower is a $A=k$-square-zero extension by $\pi_{n}(C)[n]$.
\end{prop}

\begin{proof}
Let us check that the conditions of Definition \ref{sqzero} are met for $n\geq1$ and $I=\pi_{n}C=\pi_{n}C_{\leq n}$.

\begin{enumerate}
\item Obviously $C_{\leq n}$ and $C_{\leq n-1}$ are $n$-truncated;

\item By definition of Postnikov tower, $p_n$ is an $(n-1)$-equivalence of simplicial $k$-algebras;

\item Using Remark \ref{rmk square zero extension}.1 we are reduced to show that for any $n$-truncated $k$-algebra $E$ the following diagram is homotopy cartesian:
\[
\xymatrix{
\Map_{\salg}(C_{\leq n-1},E) \ar[d] \ar[rr]^-{\Map(p_n, E)} & & \Map_{\salg}(C_{\leq n},E) \ar[d] \\
0 \ar[rr] & & \Hom_{k-\Mod}(\pi_n C, \pi_n E)
}
\]
The idea is to kill $\pi_n$ in $C_{\leq n}$ in order to obtain a better description of $C_{\le n-1}$. In order to do so, consider the following homotopy pushout in $\salg$:
\[
\xymatrix{
\mathrm{Sym}_{k}(\pi_n (C)[n]) \ar[r]^-{a} \ar[d]_-{b} & C_{\leq n} \ar[d] \\
k \ar[r] & D
}
\]
where $a$ is induced by the identity map $\pi_n C \to \pi_n C$ and $b$ is induced by the zero map $\pi_n C \to k$ via the canonical identifications
\begin{align*}
\Hom_{\salg}(\mathrm{Sym}_k(\pi_n(C)[n]),E) & \cong \Hom_{k \textrm{-} \Mod}(\pi_n(C) \otimes_k k[S^n], E) \\
& \cong \Hom_{k \textrm{-} \Mod}(\pi_n(C), \Map(k[S^n], E) \\
& \cong \Hom_{k \textrm{-} \Mod}(\pi_n(C), \pi_0 \Map(k[S^n], E)) & \text{(use Lemma \ref{lemma zero truncation})} \\
& \cong \Hom_{k \textrm{-} \Mod}(\pi_n(C), \pi_n(E))
\end{align*}
Assume for the moment that $\tau_{\leq n}D \simeq C_{\leq n-1}$ in $\mathrm{Ho}(C_{\leq n}/\salg)$; in this case, for any $n$-truncated object $E$ in $\salg$, we get
\begin{align*}
\mathrm{Map}_{\salg}(C_{\leq n-1}, E) & \simeq \mathrm{Map}_{\salg}(\tau_{\leq n}D, E) \simeq \mathrm{Map}_{\salg}(D, E) \\
& \simeq \mathrm{Map}_{\salg}(C_{\leq n}, E) \times_{\mathrm{Map}_{\salg}(\mathrm{Sym}_{k}(\pi_n (C)[n]), E) }^{h} \mathrm{Map}_{\salg}(k, E)
\end{align*}
but
\[
\mathrm{Map}_{\salg}(k, E) \simeq *
\]
and
\[
\mathrm{Map}_{\salg}(\mathrm{Sym}_{k}(\pi_n (C)[n]), E) \simeq \mathrm{Map}_{k-\Mod}(\pi_n (C)[n], E) \simeq \Map_{k-\Mod}(\pi_n C, \pi_n E)
\]
Since $\pi_n(C)$ and $\pi_n(E)$ are discrete, it follows that $\Map_{k \textrm{-} \Mod}(\pi_n C, \pi_n E)$ is discrete as well, so that there is a weak equivalence:
\[
\Map_{k\textrm{-}\Mod}(\pi_n C, \pi_n E) \simeq \pi_0 \Map_{k\textrm{-}\Mod}(\pi_n C, \pi_n E) = \Hom_{k \textrm{-} \Mod}(\pi_n C, \pi_n E)
\]
completing the proof of this step.

\item The canonical map $\pi_{n}(C_{\leq n})/I=0 \rightarrow \pi_{n}(C_{\leq n-1})= 0$ is obviously an isomorphism.
\end{enumerate}

Thus, we are left to show that there is a weak equivalence $\tau_{\leq n}D \simeq C_{\leq n-1}$ in $C_{\le n} / \salg$. To prove this, we will be using the spectral sequence of \cite[Theorem II.6.b]{Quillen_Homotopical_1967}. To begin with, set
\[
R_* := \pi_*(\mathrm{Sym}_k(\pi_n(C)[n]))
\]
so that the spectral sequence reads
\[
E^2_{p,q} = \mathrm{Tor}_p^{R_*}(\pi_*(C_{\le n}), \pi_*(k))_q \Rightarrow \pi_{p+q}(D)
\]
Choose a flat resolution $C_{\bullet *} \to \pi_*(C_{\le n})$ as $R_*$-module so that
\[
\Tor_p^{R_*}(\pi_*(C_{\le n}), \pi_*(k))_q = H^p((C_{\bullet *} \otimes_{R_*} \pi_*k)_q)
\]
Introduce the ideal
\[
I := \bigoplus_{n \ge 1} R_n
\]
Since
\[
R_q = \begin{cases}
k & \text{if } q = 0 \\ 0 & \text{if } 0 < q < n \\ \pi_n(C) & \text{if } q = n
\end{cases}
\]
it follows that
\[
\pi_* k \simeq k \simeq R_*/I
\]
and therefore
\[
C_{\bullet *} \otimes_{R_*} k \simeq C_{\bullet *} / I C_{\bullet *}
\]
In particular, being $I$ a graded ideal, we get
\[
(C_{\bullet *} \otimes_{R_*} k)_q \simeq C_{\bullet q} / J
\]
where
\[
J := \bigoplus_{i+j = q} I_i C_{\bullet j}
\]
As consequence we see that
\[
C_{\bullet q} \otimes_{R_*} k = \begin{cases}
C_{\bullet q} & \text{if } q < n \\ C_{\bullet q} / \pi_n(C) & \text{if } q = n
\end{cases}
\]
Finally, this enables us to compute the second layer of the spectral sequence:
\[
\Tor_p^{R_*}(\pi_*(C_{\le n}), k)_q = \begin{cases}
H^p(C_{\bullet q}) = \delta_{p0} \cdot \pi_q C_{\le n} & \text{if } 0 \le q < n \\
H^p(C_{\bullet n} / \pi_n(C_{\le n}) C_{\bullet 0}) & \text{if } q = n
\end{cases}
\]
In order to compute $H^0(C_{\bullet n} / \pi_n(C_{\le n}) C_{\bullet 0})$, we observe that by construction of $C_{\bullet*}$, the sequence of $R_*$-modules
\[
C_{1 *} \to C_{0 *} \to \pi_*(C_{\le n}) \to 0
\]
is exact. Now, the functor $- \otimes_{R_*} R_*/I$ is right exact and the operation of taking the degree $n$ of an $R_*$-modules defines obviously an exact functor
\[
R_* \textrm{-} \Mod \to R_0 \textrm{-} \Mod
\]
Applying these two functors to the previous exact sequence yields the new sequence
\[
C_{1 n} / \pi_n(C_{\le n}) C_{1,0} \to C_{0n} / \pi_n(C_{\le n}) C_{0,0} \to \pi_n(C_{\le n}) / \pi_n(C) \to 0
\]
which is still exact; in this way we obtain:
\[
H^0(C_{\bullet n} / \pi_n(C_{\le n}) C_{\bullet 0}) \simeq \pi_n (C_{\le n}) / \pi_n(C) = 0
\]
We finally get
\[
\pi_q(D) = \begin{cases} \pi_q(C_{\le n}) & \text{if } q < n \\ 0 & \text{if } q = n \end{cases}
\]
Moreover, the map $C_{\le n} \to D$ induces on the level of $\pi_q$ the map
\[
H^0(C_{\bullet q} \to C_{\bullet q} \otimes_{R_*} k)
\]
which is the identity if $q < n$. This shows that $C_{\le n} \to D$ is an $(n-1)$-equivalence.

At this point, consider the diagram
\[
\xymatrix{
\mathrm{Sym}_k(\pi_n(C)[n]) \ar[r]^-a \ar[d]_b & C_{\le n} \ar[d] \ar@/^1pc/[ddr]^{p_n} \\
k \ar[r] \ar@/_1pc/[drr]_\varphi & D \ar@{.>}[dr] \\ & & C_{\le n-1}
}
\]
In order to prove the existence of the dotted map, we have to show that $p_n \circ a = \varphi \circ b$; by the universal property of the symmetric algebra, this is equivalent to show that the following square commutes:
\[
\xymatrix{
\pi_n(C) \ar[r]^{\mathrm{id}} \ar[d] & \pi_n(C_{\le n}) \ar[d] \\
\pi_n(k) \ar[r] & \pi_n(C_{\le n-1})
}
\]
and since $n \ge 1$ $\pi_n(k) = \pi_n(C_{\le n-1}) = 0$, so that the last statement is trivially true.

The two-out-of-three property now shows that $D \to C_{\le n-1}$ is an $(n-1)$-equivalence; since applying $\pi_n$ we get $\pi_n(D) = \pi_n(C_{\le n-1}) = 0$, it follows that the induced map
\[
\tau_{\le n} D \to C_{\le n-1}
\]
is an $n$-equivalence, hence an equivalence in $C_{\le n} / \salg$ by the local Whitehead theorem.
\end{proof}

At this point we easily recover the important \cite[Lemma 2.2.1.1]{HAG-II}:

\begin{cor} \label{cor controlling postnikov}
Let $A \in \salg$ be a simplicial algebra. For every $n \ge 1$ there exists a unique (derived) derivation
\[
d_n \in \pi_0 \Map_{\salg / A_{\le n-1}}(A_{\le n-1}, A_{\le n-1} \oplus \pi_n(A)[n+1])
\]
such that the associated infinitesimal extension
\[
A_{\le n-1} \oplus_{d_n} \pi_n(A)[n] \to A_{\le n-1}
\]
is isomorphic in $\mathrm{Ho}(\salg/A_{n-1})$ to
\[
A_{\le n} \to A_{\le n-1}
\]
\end{cor}

\begin{proof}
Proposition \ref{prop postnikov sqzero} implies that $A_{\le n} \to A_{\le n-1}$ is a square-zero extension, so that the result follows at once from Theorem \ref{sqzerotoinf}.
\end{proof}

\begin{rmk}
In other words, Corollary \ref{cor controlling postnikov} says that for every simplicial algebra $A$, the $n$-th stage $A_{\le n}$ of its Postnikov decomposition is completely controlled by the $(n-1)$-th stage $A_{\le n-1}$, the homotopy group $\pi_n(A)$ and an element of $\overline{k_n} \in [\mathbb L_{A_{\le n-1}}, \pi_n(A)[n+1]]$ via the condition that the following is a homotopy pullback diagram:
\[
\xymatrix{
A_{\le n} \ar[d]_{p_n} \ar[r]^-{p_n} & A_{\le n-1} \ar[d]^0 \\
A_{\le n-1} \ar[r]^-{k_n} & A_{\le n-1} \oplus \pi_n(A)[n+1]
}
\]
Such derived derivation $\overline{k_n}$ is called the $n$-th Postnikov invariant of $A$.
\end{rmk}

\section{Connectivity estimates}

\begin{df}
Let $n \in \mathbb N$. A simplicial module $M$ is said to be \emph{$n$-connective} if $\pi_i M = 0$ for every $0 \le i < n$. A map of simplicial modules $f \colon M \to N$ is said to be \emph{$n$-connected} if $\hofib(f)$ is $n$-connective.
\end{df}

\begin{prop} \label{lemma connectivity}
Let $A$ be a simplicial $k$-algebra and let $M$ a $m$-connective $A$-module.
\begin{enumerate}
\item if $N$ is a $n$-connective $A$-module, then $M \otimes_A N$ is $(m+n)$-connective;
\item if $f \colon A \to B$ is a morphism of simplicial $k$-algebras such that $\pi_0(f)$ is an isomorphism, then the map $\varphi \colon M \to M \otimes_A B$ is a $m$-equivalence of simplicial $A$-modules.
\end{enumerate}
\end{prop}

\begin{proof}
We use the spectral sequence of \cite[II \S 6, Thm 6.b]{Quillen_Homotopical_1967}. Write
\[
R_* := \pi_*(A)
\]
so that the spectral sequence reads off as
\[
\Tor_p^{R_*}(\pi_*M, \pi_* N)_q \Rightarrow \pi_{p+q}(M \otimes_A N)
\]
We begin with the first statement. Choose a flat resolution $C_{\bullet *} \to \pi_* M$ and observe that we can in fact choose $C_{i,j} = 0$ for $j < m$ (just use the free resolutions given by the shifts of $R_*$). Then
\[
\Tor_p^{R_*}(\pi_* M, \pi_* N)_q = H^p((C_{\bullet *} \otimes_{R_*} \pi_* N)_q)
\]
and
\[
(C_{\bullet *} \otimes_{R_*} \pi_* N)_q = \bigoplus_{i + j = q} C_{\bullet i} \otimes_{R_*} \pi_j N
\]
Now, if $q \le m+n-1$ we have that necessarily $i < m$ or $j < n$, so that
\[
(C_{\bullet *} \otimes_{R_*} \pi_* N)_q = 0
\]
so that the spectral sequence degenerates yielding
\[
\pi_p(M \otimes_A N) = 0
\]
if $p \le m+n-1$.

We now turn to the second statement. Taking $N = B$ and $n = 0$ we see that $M \otimes_A B$ is $m$-connected, so that the map $\varphi \colon M \to M \otimes_A B$ is forcily an $(m-1)$-equivalence. We are left to compute $\pi_m(\varphi)$. However, the same computations as above show that $E^2_{0,m} = E^{\infty}_{0,m}$; since $\pi_0(A) \simeq \pi_0(B)$ it follows
\[
(C_{\bullet *} \otimes_{R_*} \pi_* B)_m = C_{\bullet,m} \otimes_{\pi_0(A)} \pi_0 B \simeq C_{\bullet,m}
\]
which implies
\[
\pi_m(M \otimes_A B) \simeq H^0(C_{\bullet,m} \otimes_{R_*} \pi_0 B) \simeq H^0(C_{\bullet,m}) \simeq \pi_m M
\]
Finally, we recall that the map $\pi_m(\varphi) \colon \pi_m M \to \pi_m(M \otimes_A B)$ can be computed as the $0$-th homology of the canonical map
\[
C_{\bullet, m} \to C_{\bullet, m} \otimes_{R_*} \pi_0 B
\]
completing the proof.
\end{proof}

\begin{cor} \label{cor connectivity sym}
Assume that $k$ is of characteristic $0$. Let $A\in \salg$ and $M \in A\textrm{-}\Mod$. If $M$ is $n$-connective $(n>0)$, then $\mathrm{Sym}^{p}_{A}(M)$ is $(pn)$-connective.
\end{cor}

\begin{proof} We may suppose that $M$ is cofibrant, so that the derived tensor product and derived symmetric powers are the usual underived ones.
Since $k$ is of characteristic $0$, the canonical map $r \colon M^{\otimes_{A} p} \to \mathrm{Sym}_A^p(M)$ has a right inverse (the antisymmetrization map) $i \colon \mathrm{Sym}_A^p(M)  \to M^{\otimes_{A} p}$ (i.e. $r \circ i$ is the identity of $\mathrm{Sym}_A^p(M)$).

Now since $M$ is $n$-connective, it follows from Prop. \ref{lemma connectivity}, that $M^{\otimes_{A} p}$ is $pn$-connective. But the composite $$\xymatrix{\pi_{i} ( \mathrm{Sym}_A^p(M)) \ar[r]^-{\pi_{i}(i)} & \pi_{i}(M^{\otimes_{A} p}) \ar[r]^-{\pi_{i}(r)} &  \pi_{i} ( \mathrm{Sym}_A^p(M)) }$$ is the identity, and therefore $\pi_{i} ( \mathrm{Sym}_A^p(M))=0$ whenever $\pi_{i}(M^{\otimes_{A} p})=0$. Hence $\mathrm{Sym}_A^p(M)$ has the same connectivity as $M^{\otimes_{A} p}$.
\end{proof}

The proof of the following theorem is precisely the translation of the one given in \cite[Theorem 8.4.3.12]{Lurie_Higher_algebra} However, the exposition given there is crystal-clear and we could not to improve it; as a consequence, we limit ourselves to sketch the outline of the proof.

\begin{thm} \label{thm main estimate}
Let $f \colon A \to B$ be a morphism in $\salg$ and $C_f :=\mathrm{hocofib}(f) \in A \textrm{-} \Mod$ its homotopy cofiber. Then there exists a canonical map $\alpha: C_f \otimes_{A} B \to \mathbb{L}_{f}$ in $\mathrm{Ho}(B-Mod)$, and we have that $\alpha$ is $(2n+2)$-connected if $f$ is $n$-connected ($n \in \mathbb{N}$).
\end{thm}

\begin{proof}
Let $\mathbb L(f) \colon \mathbb L_A \to \mathbb L_B$ be the canonical map induced by $f$, so that
\[
\mathbb L_f \simeq \hocofib (\mathbb L(f))
\]
We have a canonical map
\[
\eta \colon \mathbb L_B \to \mathbb L_f
\]
corresponding to a derived derivation
\[
d_\eta \colon B \to B \oplus \mathbb L_f
\]
Observe that $\eta \circ \mathbb L(f)$ is nullhomotopic; denote by $\varphi_0^A$ the derivation associated to the null morphism $\mathbb L_A \to \mathbb L_f$; the equivalence of simplicial sets
\[
\Map_{A \textrm{-} \Mod}(\mathbb L_A, \mathbb L_f) \simeq \Map_{\salg/A}(A, A \oplus \mathbb L_f)
\]
implies that the associated derivations, $d_{\mathbb L(f) \circ \eta}$ and $\varphi_0^A$ lie in the same path component, i.e. they are homotopic. Using the notations of Remark \ref{rmk bifunctor} and Lemma \ref{lemma naturality square zero} we obtain
\[
s(f, \mathrm{id}_{\mathbb L_f}) \circ d_{\eta \circ \mathbb L(f)} = d_\eta \circ f, \qquad \varphi_0^B \circ f = s(f, \mathrm{id}_{\mathbb L_f}) \circ \varphi_0^A
\]
where $\varphi_0^B$ denotes the derivation associated to the null map $\mathbb L_B \to \mathbb L_f$. It follows now
\[
d_\eta \circ f = s(f, \mathrm{id}_{\mathbb L_f}) \circ d_{\eta \circ \mathbb L(f)} \simeq s(f,\mathrm{id}_{\mathbb L_f}) \circ \varphi_0^A = \varphi_0^B \circ f
\]

Let at this point $\psi_\eta \colon B^\eta \to B$ be the induced infinitesimal extension, defined by the homotopy pullback
\[
\xymatrix{
B^\eta \ar[r]^{\psi'} \ar[d]_{\psi_\eta} & B \ar[d]^{\varphi_0^B} \\
B \ar[r]_{d_\eta} & B \oplus \mathbb L_f
}
\]
Since $A$ is cofibrant over $k$, Corollary \ref{cor uhmp} can be used to deduce the existence of a map $f' \colon A \to B^\eta$ such that
\[
f' \circ \psi_\eta \simeq f
\]
We obtain in this way a canonical map of $A$-modules
\[
\hocofib(f) \to \hocofib(\psi_\eta)
\]
which corresponds, under adjunction, to a canonical map
\[
\alpha_f \colon \hocofib(f) \otimes_A B \to \hocofib(\psi_\eta) \simeq  \hofib(\psi_\eta)[1] \simeq \mathbb L_f
\]
where the last isomorphism is due to Proposition \ref{prop infinitesimal fiber}. We are therefore left to show that $\alpha_f$ is $(2n+2)$-connected \footnote{The map $\alpha_f$ can be constructed also using a small generalization of the beginning of the proof of Theorem \ref{sqzerotoinf}. We leave the details to the interested reader.}. 

The proof proceeds now in several steps. The strategy is to describe the map $f$ as a finite composition
\[
f = f_{n+1} \circ \phi_{n+1} \circ \ldots \circ \phi_1
\]
in such a way that $\alpha_{f_{n+1}}$ and $\alpha_{\phi_i}$ are $(2n+2)$-connected for every $i$ (plus some other conditions), and then deduce the property from stability properties of the connectivity of construction associating $\alpha_f$ to $f$. Having outlined the strategy, we prefer to begin with these stability properties:
\begin{enumerate}
\item assume that $h = gf$; if both $f$ and $g$ are $(n-1)$-connected and moreover both $\alpha_f$ and $\alpha_g$ are $(2n+2)$-connected, then $\alpha_h$ is $(2n+2)$-connected. This is (almost) straightforward after that one gives an appropriate estimate for the map $M \otimes_A N \to M \otimes_B N$, which can be found in \cite[Lemma 8.4.3.16]{Lurie_Higher_algebra}, but which can also be obtained by the usual spectral sequence of \cite[II \S 6, Thm 6.b]{Quillen_Homotopical_1967} by carefully choosing flat resolutions;

\item assume that
\[
\xymatrix{
A \ar[r]^f \ar[d] & B \ar[d] \\ A' \ar[r]^{f'} & B'
}
\]
is a pushout square. If $\alpha_f$ is $(2n+2)$-connected then $\alpha_{f'}$ is $(2n+2)$-connected. In fact, the naturality of the construction of $\alpha_f$ shows that we have a commutative diagram
\[
\xymatrix{
\hocofib(f) \otimes_A B \ar[d] \ar[r]^-{\alpha_f} & \mathbb L_f \ar[d] \\
\hocofib(f') \otimes_{A'} B' \ar[r]^-{\alpha_{f'}} & \mathbb L_{f'}
}
\]
and now we have isomorphisms
\[
\mathbb L_{f'} \simeq \mathbb L_f \otimes_B B', \qquad \hocofib(f) \otimes_A B \otimes_B B' \simeq \hocofib(f) \otimes_A B'
\]
Moreover, the dual of Proposition \ref{prop homotopy fibre} implies $\hocofib(f) \simeq \hocofib(f')$; under this isomorphism we obtain
\[
\hocofib(f) \otimes_A B' \simeq \hocofib(f') \otimes_{A'} B'
\]
Since the functor $- \otimes_{A'} B'$ preserves cofiber sequences, it preserves fiber sequences as well, yielding
\[
\hofib(\alpha_{f'}) \simeq \hofib(\alpha_f) \otimes_B B'
\]
It is sufficient to apply now Proposition \ref{lemma connectivity}.(1) to deduce that if $\hofib(\alpha_f)$ is $(2n+2)$-connected then the same holds true for $\hofib(\alpha_{f'})$;

\item for every $n$-connected $k$-module $M$, the map $f \colon \mathrm{Sym}_k(M) \to k$ induced by the null map $M \to k$ is $(2n+2)$-connected. To prove this one first observe that there is a fiber sequence
\[
M \to 0 \to \mathbb L_{k / \mathrm{Sym}_k(M)}
\]
(this is essentially the formal computation that can be found in \cite[Proposition 8.4.3.14]{Lurie_Higher_algebra}), so that the codomain of $\alpha_f$ is $M[1]$. Next, we observe that
\[
\hofib(f) \simeq \bigoplus_{i \ge 1} \mathrm{Sym}^i(M)
\]
so that
\[
\hocofib(f) \simeq \hofib(f)[1] \simeq \bigoplus_{i \ge 1} \mathrm{Sym}^i(M[1])
\]
Finally, one checks directly that the composition
\[
M[1] \simeq \mathrm{Sym}^1(M[1])[-1] \to \bigoplus_{i \ge 1} \mathrm{Sym}^i(M[1])[-1] \xrightarrow{\alpha_f} M[1]
\]
is homotopic to the identity. This implies that
\[
\hofib(\alpha_f) \simeq \bigoplus_{i \ge 2} \mathrm{Sym}^i(M[1])
\]
Since $M[1]$ is $(n+1)$-connected, the result follows now from Corollary \ref{cor connectivity sym};

\item if $f$ is $(2n+2)$-connected, then $\alpha_f$ is $(2n+2)$-connected. Indeed, it is sufficient to observe that both $B \otimes_A \hocofib(f)$ and $\mathbb L_f$ are $(2n+2)$-connective (the first thanks to Proposition \ref{lemma connectivity}.(1) and the second thanks to general properties of the cotangent complex - see for example \cite[Lemma 8.4.3.17]{Lurie_Higher_algebra}).
\end{enumerate}

\noindent As a second step, we will need to produce a suitable factorization of the morphism $f \colon A \to B$. Let $M = \hofib(f)$. Then we have a natural map $\mathrm{Sym}_k(M) \to A$ induced by the universal property of the symmetric algebra, which enables us to form the homotopy pushout square
\[
\xymatrix{
\mathrm{Sym}_k(M) \ar[d] \ar[r]^-{\psi_1} & k \ar[d] \\ A \ar[r]^-{\phi_1} & A_1
}
\]
where $\psi_1$ is the map corresponding to the null morphism $M \to k$. This induces a morphism $f' \colon A' \to B$ such that $f_1 \circ \phi_1 \simeq f$; the 2-out-of-3 property of (local) weak equivalences readily implies that $\phi_1$ is an $(n-1)$-equivalence. Then we claim that $f_1$ is $(n+1)$-connected. In fact, if
\[
I := \bigoplus_{i \ge 1} \mathrm{Sym}^i(M)
\]
denotes the homotopy fiber of $\psi$, we obtain (using the fact that $A \otimes_{\mathrm{Sym}_k(M)} -$ preserves cofiber sequences and hence fiber sequences) the following morphism of fiber sequences:
\[
\xymatrix{
A \otimes_{\mathrm{Sym}_k(M)} I \ar[d]^g \ar[r] & A \ar@{=}[d] \ar[r] & A' \ar[d]^{f_1} \\
M \ar[r] & A \ar[r] & B
}
\]
which implies $\hocofib(f_1) \simeq \hocofib(g)[1]$. Observe now that the composition
\[
M \simeq \mathrm{Sym}^1_k(M) \to I \to A \otimes_{\mathrm{Sym}_k(M)} I
\]
is a section of $g$. It follows therefore that $\hofib(g) \simeq \hocofib(g)[-1]$ is a direct summand of $A \otimes_{\mathrm{Sym}_k(M)} I$. We therefore see that it is sufficient to show that this tensor product is $n$-connected. However, this follows at once from Proposition \ref{lemma connectivity}.(1) and Corollary \ref{cor connectivity sym}.

\noindent We are finally ready to prove that $\alpha_f$ is always $(2n+2)$-connected. Using the second step we can write $f$ as a composition
\[
f = f_{n+1} \circ \phi_{n+1} \circ \ldots \circ \phi_1
\]
where $f_{n+1}$ is $(2n+2)$-connected. Using 4. and recalling that each of the maps $\phi_i$ is $(n-1)$-connected, we can use 1. to reduce ourselves to prove that $\alpha_{\phi_i}$ is $(2n+2)$-connected for every $i$. However, this follows from 2. and 3.
\end{proof}

\begin{cor}\label{cor1}
Let $f:A\to B$ be a map in $\salg$, and $n\in \mathbb{N}$.
\begin{enumerate}
\item If $f$ is $n$-connected, then $\mathbb{L}_{f}$ is $(n+1)$-connective.
\item If  $\mathbb{L}_{f}$ is $(n+1)$-connective and $\pi_0(f)$ is an isomorphism, then $f$ is $n$-connected.
\end{enumerate}
\end{cor}

\begin{proof}
The first part is an immediate consequence of Theorem \ref{thm main estimate}. In fact, using the notations of that theorem, if $f$ is $n$-connected, then $C_f = \hofib(f)[1]$ is $(n+1)$-connective and $\alpha : C_f \otimes_{A} B \to \mathbb{L}_{f}$ is $(2n+2)$-connected; moreover, Proposition \ref{lemma connectivity}.(1) implies that $C_f \otimes_A B$ is at least $(n+1)$-connective; the long exact sequence associated to a cofiber sequence implies then that $\mathbb L_f$ is $n$-connective.

Conversely, assume that $\pi_0(f)$ is an isomorphism. We will show that if $f$ is not $n$-connected, then $\mathbb L_f$ is not $(n+1)$-connective. We can assume that $n$ is minimal with respect to this property, so that $f$ is $(n-1)$-connected and $\pi_n C_f \ne 0$. Observe that $A \to B \to C_f$ is a fiber - cofiber sequence, and therefore $\pi_0(B) \to \pi_0(C_f)$ is surjective. If $\pi_0(f)$ is an isomorphism, we obtain $\pi_0(C_f) = 0$, so that $n \ge 1$.

Since $C_f$ is $n$-connected, the map
\[
\alpha \colon C_f \otimes_A B \to \mathbb L_f
\]
is $(2n)$-connected. Since $2n > n$, it follows that
\[
\pi_n(C_f \otimes_A B) \to \pi_n \mathbb L_f
\]
is an isomorphism. Moreover, since $\pi_0(A) \simeq \pi_0(B)$ we can apply Proposition \ref{lemma connectivity}.(2) to conclude that the map
\[
\pi_n C_f \to \pi_n(C_f \otimes_A B)
\]
is an isomorphism as well. It follows that
\[
\pi_n \mathbb L_f \simeq \pi_n C_f \ne 0
\]
completing the proof.
\end{proof}

\begin{rmk}\begin{enumerate}
\item Note that the proof of Corollary \ref{cor1} (1), shows a bit more than what is in the statement. In the fiber sequence $$\xymatrix{C_f \otimes_{A} B \ar[r]^-{\alpha} & \mathbb{L}_{f} \ar[r] & \mathrm{hocofib}(\alpha), }$$ we know that $\mathrm{hocofib}(\alpha)$ is $(2n+3)$-connective and that $C_f \otimes_{A} B$ is $(n+1)$-connective. Therefore we may also identify the first a priori non-zero homotopy group of $\mathbb{L}_{f}$:  $\pi_{n+1} (C_f \otimes_{A} B) \simeq \pi_{n+1}(\mathbb{L}_{f})$ (since $2n+3 > n+2$ for $n\geq 0$). More generally, we have that the $i$-th homotopy groups of $C_f \otimes_{A} B$ and $\mathbb{L}_{f}$ are isomorphic for any $i < 2n+2$ (the interest of this remark grows linearly with $n$).
\item It follows from the previous corollary that the relative cotangent complex $\mathbb{L}_{\pi_0 (A)/A}$ is $1$-connective (i.e. $\pi_{i}\mathbb{L}_{\pi_0 (A)/A}=0$ for $i=0,1$). So the same is true for $\mathbb{L}_{\textrm{t}(X)/X}$ where $X$ is a Deligne-Mumford derived stack and $\textrm{t}(X)$ its truncation.
\end{enumerate}
\end{rmk}

\begin{cor}
For a morphism $f:A\to B$ in $\salg$ the following properties are equivalent
\begin{enumerate}
\item $f$ is a weak equivalence
\item $\pi_{0}(f):\pi_0 (A) \to \pi_0 (B)$ is an isomorphism, and $\mathbb{L}_{f} \simeq 0$.
\end{enumerate}
\end{cor}

\begin{proof}
$(1)\Rightarrow (2)$ is obvious. From Corollary \ref{cor1}, we get that $f$ is $n$-connected for any $n\geq 0$, i.e. it is a weak equivalence. So $(2)\Rightarrow (1)$.
\end{proof}

\section{An exercise in derived deformation theory}
We want to explain how derived deformation theory fills the gaps in classical deformation theory, by working out an explicit example of a very 'classical' deformation problem: the \emph{infinitesimal deformations of a proper smooth scheme over} $k=\mathbb{C}$. 

Since we will work out an example in characteristic zero, the reader might, in this \S , switch from $\salg$ to $\mathbf{cdga}^{\leq 0}_{\mathbb{C}}$, if he wishes to.\\
 
Let us recall that the object of study of classical (formal) deformation theory are reduced functors
\[
 \mathbf{Art}_{\mathbb C} \to \mathbf{Grpd} \hookrightarrow \sset
\]
(i.e. the image of $\mathbb C$ is weakly contractible). Here, $\mathbf{Art}_{\mathbb C}$ denote the category of artinian $\mathbb C$-algebras with residue field isomorphic to $\mathbb C$. For example, if $F \colon \mathbf{Alg}_{\mathbb C} \to \mathbf{Grpd}$ is a classical moduli problem and $\xi \in F(\mathbb C)$ is a point, we can obtain a formal reduced functor by forming the homotopy pullback
\[
\hat{F}_\xi := F \times_{F(\mathbb C)} \xi
\]
and then restricting it to $\mathbf{Art}_{\mathbf C}$; this is called the \emph{formal completion} of $F$ at $\xi$.\\

An interesting, and classically well known moduli functor is given by
\[
F \colon \mathbf{Alg}_{\mathbb C} \to \mathbf{Grpd}
\]
sending a $\mathbb C$-algebra $R$ into the groupoid of proper smooth morphisms
\[
Y \to \mathrm{Spec}(R)
\]
and isomorphisms between them. In this case, if we fix a proper smooth scheme 
\[
\xi \colon X_0 \to \mathrm{Spec}(\mathbb C)
\]
the corresponding homotopy base change $\hat{F}_\xi$ is exactly the usual functor $\mathrm{Def}_{X_0}$. The following properties are well known:

\begin{enumerate}
\item $\hat{F}_\xi(\mathbb C[t]/t^{n+1})$ is the groupoid of $n$-th order infinitesimal deformations of $\xi$;

\item if $\xi_1 \in \hat{F}_\xi(\mathbb C[\varepsilon])$ is a first order deformation of $\xi$, then $\mathrm{Aut}_{\hat{F}_\xi(\mathbb C[\varepsilon])}(\xi_1) \simeq H^0(X_0, T_{X_0})$;

\item $\pi_0( \hat{F}_\xi(\mathbb C[\varepsilon])) \simeq H^1(X_0, T_{X_0})$;

\item if $\xi_1$ is a first order deformation there exists an obstruction $\mathrm{obs}(\xi_1) \in H^2(X_0, T_{X_0})$ such that $\mathrm{obs}(\xi_1) = 0$ if and only if $\xi_1$ extends to a second order deformation.
\end{enumerate}

The first three properties are really satisfactory, but not the fourth one. It raises two \emph{questions} :
\begin{enumerate}
\item \emph{how to interpret geometrically the entire $H^2(X_0, T_{X_0})$ ?}
\item \emph{how to identify intrinsically the space of all obstructions\footnote{It can happen that every obstruction is trivial and $H^2(X_0,T_{X_0}) \ne 0$. An example is given by a smooth projective surface $X_0 \subseteq \mathbb P_{\mathbb C}^3$ of degree $\ge 6$.} inside $H^2(X_0, T_{X_0})$ ?} 
\end{enumerate}

Derived deformation theory gives a more general perspective on the subject, and answers both questions. It allows a natural interpretation of $H^2(X_0,T_{X_0})$ as the group of derived deformations i.e. (isomorphism classes of) deformations over a specific non-classical ring, and it identifies, consequently, the obstructions space in a very natural way. Let's work these answers out.

Define
\[
\underline{F} \colon \mathrm{sAlg}_{\mathbb C} \to \sset
\]
sending a simplicial algebra $A$ to the nerve of the category of proper smooth maps of derived schemes
\[
Y \to \mathbb R \mathrm{Spec}(A)
\]
and equivalences between them. Now, it is easy to check that if $Y \to \mathrm{Spec} \, R$ is a smooth map from $Y$ a derived scheme to a usual affine scheme, then $Y$ is underived (i.e. equivalent to its truncation). As a consequence, we get $\underline{F}(R)= F(R)$, for any discrete $R$. It can be shown that $F$ is in fact a derived Artin stack, hence it is infinitesimally cohesive (see \cite[Corollary 6.5]{DAG-IX} and \cite[Lemma 2.1.7]{DAG-XIV}) i.e.\  it preserves homotopy pullbacks diagrams of the form $$\xymatrix{A' \ar[r] \ar[d] & A \ar[d]^-{p}\\ B' \ar[r]_-{q} & B }$$ with $\pi_0(p)$ and $\pi_0(q)$ surjective with nilpotent kernels. Introduce the full subcategory $\mathrm{sArt}_{\mathbb C}$ of $\mathrm{sAlg}_{\mathbb C}$  of simplicial $\mathbb C$-algebras $A$ such that $\pi_0(A) \in \mathbf{Art}_{\mathbb C}$, each $\pi_{i}(A)$ is a $\pi_0 (A)$-module of finite type, and $\pi_i(A)=0$, for $i>>0$. If we fix
\[
\xi \in \underline{F}(\mathbb C) = F(\mathbb C)
\]
then we can, as above, form the derived completion of $\underline{F}$ at $\xi$ by taking the homotopy pullback:
\[
\hat{\underline{F}}_\xi := \underline{F} \times_{\underline{F}(\mathbb C)}^h \xi
\]
The following proposition answers to Question 1 above, by saying that the entire $H^2(X_0,T_{X_0})$ can be interpreted as a space of deformations over a \emph{derived} base.

\begin{prop} \label{prop deformation}
$\pi_0(\hat{\underline{F}}_\xi (\mathbb C \oplus \mathbb C[1])) \simeq H^2(X_0,T_{X_0})$.
\end{prop}

\begin{proof}
First of all, $\underline{F}$ has a cotangent complex at $\xi$ in the sense of \cite[Definition 1.4.1.5]{HAG-II} and it can be shown that
\[
\mathbb T_{\underline{F},\xi} \simeq \mathbb R\Gamma(\mathbb T_{X_0}[1])
\]
Using \cite[Proposition 1.4.1.6]{HAG-II} we obtain
\[
\mathbb L_{\underline{F}, \xi} \simeq \mathbb T_{\underline{F},\xi}^* \simeq \mathbb R \Gamma(X_0, \mathbb L_{X_0}[-1])
\]
Therefore
\begin{align*}
\pi_0(\underline{\mathrm{Der}}_{\underline{F}}(\xi ; \mathbb C[1])) & \simeq \pi_0( \mathbb R \underline{\Hom}_{\mathbb C}(\mathbb L_{\underline{F},\xi}, \mathbb C[1])) \\
& \simeq \mathrm{Ext}^0(\mathbb L_{\underline{F},\xi}, \mathbb C[1]) \\
& = \mathrm{Ext}^1(\mathbb L_{\underline{F},\xi}, \mathbb C) \\
& = \mathrm{Ext}^0(\mathbb L_{\underline{F},\xi}[-1], \mathbb C) \\
& = \mathbb T_{\underline{F},\xi}[1] \simeq \mathbb R \Gamma(X_0, \mathbb T_{X_0}[2])
\end{align*}
since $X_0$ is smooth, we obtain $\mathbb T_{X_0} \simeq T_{X_0}$, so that
\[
\mathbb R \Gamma(X_0, \mathbb T_{X_0}[2]) \simeq H^2(X_0, T_{X_0})
\]
In conclusion
\begin{align*}
\pi_0(\hat{\underline{F}}_\xi(\mathbb C \oplus \mathbb C[1])) & \simeq \pi_0 (\hofib (\underline{F}(\mathbb C \oplus \mathbb C[1]) \to \underline{F}(\mathbb C), \xi)) \\
& \simeq \pi_0(\underline{\mathrm{Der}}_{\underline{F}}(\xi ; \mathbb C[1])) \simeq H^2(X_0,T_{X_0})
\end{align*}
\end{proof}

Now that we have a derived deformation interpretation of $H^2(X_0,T_{X_0})$ at hand, we can proceed by answering Question 2 above. We begin by the following

\begin{lem}
Let
\[
\xymatrix{
I \ar[r] & A' \ar[r]^f & A
}
\]
be a square zero extension of (augmented) artinian $\mathbb C$-algebras (i.e. $I^2 = 0$). Then there exist a derivation $d \colon A \to A \oplus I[1]$ and a homotopy cartesian diagram
\[
\xymatrix{
A' \ar[r] \ar[d] & A \ar[d]^{\pi \circ d} \\
\mathbb C \ar[r] & \mathbb C \oplus I[1]
}
\]
where $\pi \colon A \oplus I[1] \to \mathbb C \oplus I[1]$ is the natural map induced by the augmentation $A \to \mathbb C$.
\end{lem}

\begin{proof}
Use Theorem \ref{sqzerotoinf} to deduce the existence of a derivation $d \colon A \to A \oplus I[1]$ such that
\[
\xymatrix{
A' \ar[r] \ar[d] & A \ar[d]^{\varphi_0} \\
A \ar[r]^-{d} & A \oplus I[1]
}
\]
is a homotopy pullback. We are left to show that
\[
\xymatrix{
A \ar[d]_{\varphi_0} \ar[d] \ar[r] & \mathbb C \ar[d] \\
A \oplus I[1] \ar[r] & \mathbb C \oplus I[1]
}
\]
is a homotopy pullback. However, the map $A \oplus I[1] \to \mathbb{C} \oplus I[1]$ is a fibration, hence it is sufficient to show that it is a pullback, and this is straightforward verification.
\end{proof}

Let us consider, in particular, the (discrete) square-zero extension $A'=\mathbb C[s]/(s^3) \to \mathbb C[s]/(s^2)=A$. By the previous Lemma, we obtain a homotopy pullback
\[
\xymatrix{
\mathbb C[s]/(s^3) \ar[r] \ar[d] & \mathbb C[s]/(s^2) \ar[d]^-{p} \\
\mathbb C \ar[r]_-{q} & \mathbb C \oplus \mathbb C[1]
}
\]
Since $F$ is infinitesimally cohesive, and $q: \mathbb{C} \to \mathbb C \oplus \mathbb C[1]$ and $p: \mathbb C[s]/(s^2) \to \mathbb C \oplus \mathbb C[1]$, are surjective on $\pi_0$ with nilpotent kernels, we get a homotopy pullback of simplicial sets
\[
\xymatrix{
F(\mathbb C[s]/(s^3))= \underline{F}(\mathbb C[s]/(s^3)) \ar[r] \ar[d] & \underline{F}(\mathbb C[s]/(s^2))= F(\mathbb C[s]/(s^2)) \ar[d] \\
F(\mathbb C)=\underline{F}(\mathbb C) \ar[r] & \underline{F}(\mathbb C \oplus \mathbb C[1])
}
\]
Via augmentations, this diagram maps to $F(\mathbb C)$, and if we take the four corresponding homotopy fibers at $\xi$, we get a homotopy pullback of pointed simplicial sets

\[
\xymatrix{
\hat{F}_\xi(\mathbb C[s]/(s^3))= \hat{\underline{F}}_\xi(\mathbb C[s]/(s^3)) \ar[r] \ar[d] & \hat{\underline{F}}_\xi(\mathbb C[s]/(s^2))= \hat{F}_\xi (\mathbb C[s]/(s^2)) \ar[d] \\
\bullet = \hat{F}_\xi (\mathbb{C})=\hat{\underline{F}}_\xi (\mathbb C) \ar[r] & \hat{\underline{F}}_{\xi} (\mathbb{C} \oplus \mathbb{C} [1]) }
\]

i.e. a fiber sequence of pointed simplicial sets
\[
\hat{F}_\xi(\mathbb C[s]/(s^3)) \to \hat{F}_\xi(\mathbb C[s]/(s^2)) \to \hat{\underline{F}}_\xi(\mathbb C \oplus \mathbb C[1])
\]
Then, Proposition \ref{prop deformation} allows then to write the long exact sequence:
\[
\pi_0(\hat{F}_\xi(\mathbb C[s]/(s^3))) \to \pi_0(\hat{F}_\xi(\mathbb C[s]/(s^2))) \to \pi_0(\hat{\underline{F}}_\xi(\mathbb C \oplus \mathbb C[1])) \simeq H^2(X_0,T_{X_0})
\]
of pointed sets (note that the middle and the rightmost ones are vector spaces). As a consequence, we see that a first order deformation extends to a second order deformation if and only if its image in $H^2(X_0,T_{X_0})$ vanishes. In other words, the space $\mathsf{Obs}$ of all obstructions is given by the image of the obstruction map $$\mathrm{obs}: \pi_0(\hat{F}_\xi(\mathbb C[s]/(s^2))) \to \pi_0(\hat{\underline{F}}_\xi(\mathbb C \oplus \mathbb C[1])) \simeq H^2(X_0,T_{X_0}). $$
We have therefore answered Question 2, too.\\

\noindent \textbf{Exercise.}  Extend the previous arguments to higher order infinitesimal deformations and obstructions.

\appendix

\section{Homotopical nonsense}

\subsection{Homotopy pullbacks}

The first technique we want to recall is how to compute homotopy pullbacks in a general model category. Recall first of all the following result:

\begin{prop} \label{prop right proper}
Let $\mathcal M$ be a right proper model category. If we have a diagram
\[
\xymatrix{
X \ar[r]^g & Z & Y \ar[l]_h
}
\]
where at least one of $g$ and $h$ is a fibration, then the pullback $X \times_Z Y$ is naturally weakly equivalent to the homotopy pullback.
\end{prop}

\begin{proof}
See \cite[Corollary 13.3.8]{Hirschhorn_Model_2003}
\end{proof}

We can obtain a similar result for general model categories adding the hypothesis that every object $X$, $Y$ and $Z$ is fibrant. To see this, recall first of all the following proposition:

\begin{prop} \label{prop right proper 2}
Let $\mathcal M$ be a model category and let
\[
\xymatrix{
A \ar[d] \ar[r]^u & X \ar[d]^p \\ B \ar[r]^w & Y
}
\]
be a pullback. If $p$ is a fibration and $w$ is a weak equivalence between fibrant objects, then $u$ is a weak equivalence.
\end{prop}

\begin{proof}
There is a simple argument due to Reedy (cfr. \cite[Proposition 13.1.2]{Hirschhorn_Model_2003}), but there is also a more elaborate proof that avoid any lifting argument and therefore can be carried out in the more general context of categories of fibrant objects (see \cite[\S II.8]{Goerss_Jardine_Simplicial_1999}).
\end{proof}

\begin{cor} \label{cor computing homotopy pullback}
Let $\mathcal M$ be a model category. Suppose given a pullback diagram
\[
\xymatrix{
A \ar[r] \ar[d] & B \ar[d]^p \\ C \ar[r] & D
}
\]
where $B$, $C$ and $D$ are fibrant objects and $p$ is a fibration. Then the square is a homotopy pullback.
\end{cor}

\begin{proof}
The same proof of Proposition \ref{prop right proper} applies, because the only needed fact is the stability of weak equivalences under pullback by fibrations, and this is guaranteed by Proposition \ref{prop right proper 2}.
\end{proof}

We conclude describing the ``universal homotopy mapping property'' of the pullback that everyone could imagine (but for which we don't have any written reference):

\begin{cor} \label{cor uhmp}
Let $\mathcal M$ be a model category and let
\[
\xymatrix{
A \ar[d]_{f'} \ar[r]^{g'} & B \ar[d]^f \\ C \ar[r]_g & D
}
\]
be a homotopy pullback in $\mathcal M$. If $X$ is a cofibrant object and $\alpha \colon X \to B$, $\beta \colon X \to C$ are morphisms such that $f \circ \alpha \simeq g \circ \beta$, then there is a map $\gamma \colon X \to A$ such that $g' \circ \gamma \simeq \alpha$ and $f' \circ \gamma \simeq \beta$. Moreover, such $\gamma$ is unique up to homotopy.
\end{cor}

\begin{proof}
We can assume $B,C$ and $D$ to be fibrant and the maps $f$ and $g$ to be fibrations. In this case, use the cofibrancy of $X$ to choose a cylinder object 
\[
\xymatrix{
X \sqcup X \ar[r]^{(i_0,i_1)} & \mathrm{Cyl}(X) \ar[r]^w & X
}
\]
for $X$ and a homotopy $H \colon \mathrm{Cyl}(X) \to D$ such that
\[
\xymatrix{
X \ar[d]_{i_0} \ar@/^.5pc/[dr]^{f \circ \alpha} \\ \mathrm{Cyl}(X) \ar[r]^-H & D \\
X \ar[u]^{i_1} \ar@/_.5pc/[ur]_{g \circ \beta}
}
\]
commutes. The liftings in the diagrams
\[
\xymatrix{
X \ar[r]^\beta \ar[d]_{i_1} & C \ar[d]^g \\ \mathrm{Cyl}(X) \ar@{.>}[ur]^{K_1} \ar[r]_H & D
}, \qquad
\xymatrix{
X \ar[r]^\alpha \ar[d]_{i_0} & B \ar[d]^f \\ \mathrm{Cyl}(X) \ar@{.>}[ur]^{K_2} \ar[r]_H & D
}
\]
exist because $i_0$ and $i_1$ are trivial cofibrations, while $f$ and $g$ are fibrations by assumption. In particular we get
\[
f \circ K_2 = H = g \circ K_1
\]
which produces a unique map $\delta \colon X \times I \to A$. Set
\[
\gamma := \delta \circ i_0
\]
We therefore have
\[
g' \circ \gamma = g' \circ \delta \circ i_0 = K_2 \circ i_0 = \alpha
\]
and
\[
f' \circ \gamma = f' \circ \delta \circ i_0 = K_1 \circ i_0 \simeq K_1 \circ i_1 = \beta
\]
The uniqueness up to homotopy of $\gamma$ is easily seen with a similar construction.
\end{proof}

\subsection{Homotopy fibres}

\begin{prop} \label{prop homotopy fibre}
Let $\mathcal M$ be a pointed model category and let
\[
\xymatrix{
A \ar[r]^\alpha \ar[d]_f & B \ar[d]^g \\ C \ar[r]^\beta & D
}
\]
be a given homotopy pullback. Then $\hofib f \simeq \hofib g$.
\end{prop}

\begin{proof}
We can compute an explicit model for the homotopy pullback by replacing $B$, $C$ and $D$ by fibrant objects and the maps $g$ and $\beta$ by fibrations. This means that we can assume from the beginning that $g$ and $\beta$ are fibrations between fibrant objects. Then $f$ is a fibration as well and $\hofib f$ is defined to be the homotopy pullback
\[
\xymatrix{
\hofib f \ar[r] \ar[d] & A \ar[d]^f \\ {*} \ar[r] & C
}
\]
Since $C$, $*$ and $A$ are fibrant and $f$ is a fibration it follows from Corollary \ref{cor computing homotopy pullback} that the (strict) pullback of the maps $* \to C \leftarrow A$ is an explicit model for the homotopy fiber. It follows that the outer rectangle in
\[
\xymatrix{
\hofib f \ar[r] \ar[d] & A \ar[d]^f \ar[r] & B \ar[d]^g \\ {*} \ar[r] & C \ar[r] & D
}
\]
is a pullback, hence (for the same reason as above) a homotopy pullback, showing that
\[
\hofib f \simeq \hofib g
\]
\end{proof}

\section{Homotopy of Simplicial rings}

\subsection{Simplicial algebras}

Throughout this section we will denote by $k$ a fixed field and we will denote by $\salg$ the category of simplicial objects in $\mathbf{Alg}_k$. The canonical adjunction
\[
\mathrm{Sym}_k \colon \sset \leftrightarrows \salg \colon \mathcal U
\]
where $\mathcal U$ is the obvious forgetful functor satisfies the hypothesis of the transfer principle, so that we can endow $\salg$ with a model structure where
\begin{enumerate}
\item a map $f \colon A \to B$ is a weak equivalence or a fibration if and only if the map $\mathcal U(f)$ is so;
\item a map $f \colon A \to B$ is a cofibration if and only if it has the left lifting property with respect to every trivial fibration.
\end{enumerate}

We have a natural inclusion $i \colon \mathbf{Alg}_k \to \salg$ which defines a reflective subcategory. In fact one has the following:

\begin{lem} \label{lemma zero truncation}
The functor $\pi_0 \colon \salg \to \mathbf{Alg}_k$ is left adjoint to the inclusion functor $i$.
\end{lem}

\begin{proof}
Let $A$ be any simplicial $k$-algebra and consider the $k$-algebra $\pi_0(A)$. We clearly have a morphism
\[
\eta_A \colon A \to \pi_0(A)
\]
defined by sending an $n$-simplex $a \in A_n$ into the path component of any of its vertices. The compatibility with the sum and the product is a natural consequence of the fact that the face maps of $A$ are compatible with the algebra structure (i.e. $d_n \colon A_n \to A_{n-1}$ is a morphism of $k$-algebras).

If $B$ is any discrete $k$-algebra and $\varphi \colon A \to B$ is any morphism we immediately obtain a morphism of $k$-algebras
\[
\pi_0(\varphi) \colon \pi_0(A) \to \pi_0(B) = B
\]
which moreover satisfies $\pi_0(\varphi) \circ \eta_A = \varphi$. The uniqueness of $\pi_0(\varphi)$ is clear, hence it follows that $\pi_0 \dashv i$ by the standard characterization of the adjuctions via the universal property of the unit.
\end{proof}

\subsection{Modules over simplicial rings}

Let $A \in \salg$ be a fixed simplicial $k$-algebra. The category of (simplicial) $A$-modules, denoted $A\textrm{-} \Mod$ inherits a model structure from $\salg$ using the classical result that can be found in \cite{Schwede_Shipley_Algebras_2000}. This category is naturally endowed with a forgetful functor
\[
A \textrm{-} \Mod \to \sset
\]
which is right adjoint to
\[
A[-] \colon \sset \to A \textrm{-} \mathbf{Mod}
\]

\begin{df}
Let $A$ be a simplicial $k$-algebra. For every $A$-module $M$ and any positive integer $n \ge 0$ set
\[
M[n] := M \otimes_A A[S^n]
\]
where $S^n$ is a simplicial model for the $n$-sphere.
\end{df}

If $M$ is an $A$-module, we can define its homotopy groups simply using the forgetful functor to $\sset$. With this definition one immediately obtains the following lemma:

\begin{lem} \label{lemma homotopy groups of simplicial modules}
For any $A$-module $M$ it holds
\[
\pi_n(N) \simeq \pi_0 \Map_{A \textrm{-} \Mod}(A[S^n],N)
\]
\end{lem}

\begin{proof}
One has to observe that setting $M \otimes K := M \otimes_A A[K]$ for any $A$-module $M$ and any simplicial set $K$ define a tensor over $\sset$ which is in fact part of a simplicial model structure over $A \textrm{-} \Mod$ (see for example \cite[Chapter II.4]{Quillen_Homotopical_1967}). It follows that
\[
\Map_{A \textrm{-} \Mod}(A[S^n],N) \simeq \Map_{\sset}(S^n, N)
\]
and now the thesis follows by definition of $\pi_n(N)$.
\end{proof}

Since $A \textrm{-} \Mod$ is a pointed model category, it follows that we can define a suspension and a loop functor. More precisely, we consider the following definition:

\begin{df}
Let $M$ be an $A$-module. The suspension of $M$ is defined to be the homotopy pushout
\[
\xymatrix{
M \ar[r] \ar[d] & 0 \ar[d] \\ 0 \ar[r] & \Sigma(M)
}
\]
\end{df}

We define the loop functor in a similar way:

\begin{df}
Let $M$ be an $A$-module. The loop of $M$ is defined to be the homotopy pullback
\[
\xymatrix{
\Omega(M) \ar[r] \ar[d] & 0 \ar[d] \\ 0 \ar[r] & M
}
\]
\end{df}

With these definitions, we can prove that $A \textrm{-} \Mod$ is ``almost stable'', in the sense that $\Sigma$ is not an equivalence, but $\Omega \Sigma(M) \simeq M$ for any simplicial module $M$. The result is essentially due to Quillen, see \cite[Proposition II.6.1]{Quillen_Homotopical_1967}. We will need a preliminary result on the form of cofibrations of $A \textrm{-} \Mod$.

\begin{df}
A map $f \colon M \to N$ in $A \textrm{-} \Mod$ is said to be \emph{free} if there are subsets $C_q \subset N_q$ for each $q \in \N$ such that:
\begin{enumerate}
\item $\eta^* C_p \subseteq C_q$ whenever $\eta \colon \mathbf q \to \mathbf p$ is a surjective monotone map;
\item for every $q \in \N$ the map $(f_q,g_q) \colon M_q \oplus A[C_q] \to N_q$ is an isomorphism, where $g_q \colon A[C_q] \to N_q$ is the map induced by the inclusion $C_q \subseteq N_q$.
\end{enumerate}
\end{df}

\begin{rmk}
A free morphism $f \colon M \to N$ in $A \textrm{-} \Mod$ is always degreewise injective. In  fact, $M_q \to \oplus A[C_q]$ is injective, so that $f_q \colon M_q \to N_q$ is injective for each $q \in \N$.
\end{rmk}

\begin{prop} \label{prop cofibrations are injective}
A morphism $f \colon M \to N$ in $A \textrm{-} \Mod$ is a cofibration if and only if it is a retract of a free map. In particular, every cofibration in $A \textrm{-} \Mod$ is degreewise injective.
\end{prop}

\begin{proof}
See \cite[Remark 4, page II.4.11]{Quillen_Homotopical_1967} for a proof that every free map is a cofibration. The small object argument can be used to show that every map $f$ admits a factorization as $f = p i$, where $p$ is a trivial fibration and $i$ is a free map. It follows that if $f$ is a cofibration, then it is a retract of a free map. The second statement follows at once, since the retract of an injective map is still an injective map.
\end{proof}

\begin{cor} \label{cor omega sigma}
Let $A$ be a simplicial $k$-algebra. Then for any $A$-module $M$ there is a weak equivalence $\Omega \Sigma(M) \simeq M$.
\end{cor}

\begin{proof}
We have to show that if the square
\[
\xymatrix{
M \ar[r] \ar[d] & 0 \ar[d] \\ 0 \ar[r] & N
}
\]
is a homotopy pushout, then it is also a homotopy pullback. First of all, we can suppose without loss of generality that $M$ is a cofibrant object; next, we can replace the map $M \to 0$ with a cofibration $j \colon M \to D$ where $D$ is weakly equivalent to $0$. The dual of Corollary \ref{cor computing homotopy pullback} shows that the pushout
\[
\xymatrix{
M \ar[d] \ar[r]^j & D \ar[d]^p \\ 0 \ar[r] & N'
}
\]
is an explicit model for the suspension of $M$. In other words, we have
\[
\Sigma(M) \simeq N' := \mathrm{coker}(j)
\]
Now, $N'$ and $D$ are fibrant objects and $p \colon D \to N'$ is a surjective map, hence it is a fibration. It follows again from Corollary \ref{cor computing homotopy pullback} that $\ker(p)$ is an explicit model for $\Omega(N')$. Since Proposition \ref{prop cofibrations are injective} implies that $j$ is injective, we see that
\[
M \simeq \ker(j) \simeq \Omega \Sigma(M)
\]
completing the proof.
\end{proof}

\subsection{Derived derivations and cotangent complex}

Recall the following definition of derived derivation:

\begin{df}
Let $A$ be a simplicial $k$-algebra and let $B$ be an $A$-algebra. An $A$-derivation of $B$ with values in a $B$-module $M$ is a section of $B \oplus M \to B$, where $B \oplus M$ is defined by performing the classical square-zero extension degreewise.
\end{df}

\begin{rmk} \label{rmk bifunctor}
Fix two simplicial $k$-algebras $A$ and $C$. The previous definition gives rise to a bifunctor
\[
s \colon A / \salg / B \times B \textrm{-} \mathbf{Mod} \to A / \salg / B
\]
defined by
\[
s \colon (A \to C \to B,M) \mapsto C \oplus M
\]
where $M$ is thought as $C$-module by forgetting along the given map $C \to B$.

We have also another functor
\[
\pi \colon A / \salg / B \times B \textrm{-} \mathbf{Mod} \to A / \salg / B
\]
defined simply by
\[
\pi \colon (A \to C \to B, M) \mapsto A \to C \to B
\]
Finally, we have a natural transformation $p \colon s \to \pi$ which assigns to the pair $(C,M)$ in $A / \salg / B \times B \textrm{-} \Mod$ the natural projection
\[
C \oplus M \to C
\]
We will denote by abuse of notation this map $p_C$ (instead of $p_{C,M}$). These are easy checks left to the reader.
\end{rmk}

The set of $A$-derivations of $B$ into $M$ is naturally endowed with a $k$-module structure, which allows to define a functor
\[
\mathrm{Der}_A(B,-) \colon B \textrm{-} \Mod \to k \textrm{-} \Mod
\]
We can see this functor as the $\pi_0$ of another, much more interesting functor
\[
\mathscr Der_A(B,-) \colon B \textrm{-} \Mod \to A \textrm{-} \Mod
\]
defined by
\[
\mathscr Der_A(B,M) := \Map_{A / \salg / B}(B, B \oplus M)
\]

\begin{lem}
The functor $\mathscr Der_A(B,-)$ is representable by a simplicial $B$-module $\mathbb L_{B/A}$. In particular, it is a left Quillen functor.
\end{lem}

\begin{proof}
Let $Q(B)$ be a cofibrant replacement for $B$ in the model category $A/\salg$. Define
\[
\mathbb L_{B/A} := \Omega^1_{Q(B)/A} \otimes_{Q(B)} B
\]
where the construction of $\Omega^1_{Q(B)/A}$ is meant to be performed degreewise. It can be checked that $\mathbb L_{B/A}$ is the desired representative (see for example \cite[Chapter I.1]{HAG-II}).

The second part of the statement follows from the fact that $\Map_B(\mathbb L_{B/A},-)$ is right adjoint to $-\otimes_B \mathbb L_{B/A}$ and the fact that $\Map_B(\mathbb L_{B/A},-)$ respects fibrations and trivial fibrations (since it is defined as the internal hom of $\sset$).
\end{proof}

\begin{lem} \label{lemma naturality square zero}
Let $f \colon A \to B$ be a morphism of $C$-algebras and let $g \colon M \to N$ be a morphism of $B$-module. Any commutative triangle of $A$-modules
\[
\xymatrix{
\mathbb L_{A/C} \ar[d]_{\mathbb L(f)} \ar[r]^u & M \ar[d]^g \\ \mathbb L_{B/C} \ar[r]_v & N
}
\]
gives rise to a commutative diagram of $C$-algebras as follows:
\[
\xymatrix{
A \ar[r]^-{d_u} \ar[d]_f & A \oplus M \ar[d]^{s(f,g)} \\ B \ar[r]^-{d_v} & B \oplus N
}
\]
where $s(f, g)$ denotes the bifunctor of Remark \ref{rmk bifunctor} and $d_u$, $d_v$ are the $C$-derivation induced by the universal property of the cotangent complexes $\mathbb L_{A/C}$ and $\mathbb L_{B/C}$.
\end{lem}

\begin{proof}
Using the notations of Remark \ref{rmk bifunctor}, we see that $d_u$ is a section of $p_A$, and moreover naturality yields
\[
f \circ p_A = p_B \circ s(f,g)
\]
Since $p_A$ is an epimorphism, we conclude that the equality
\[
s(f,g) \circ d_u = d_v \circ f
\]
holds if and only if
\[
s(f,g) \circ d_u \circ p_A = d_v \circ f \circ p_A
\]
and now this follows from the already stated properties.
\end{proof}

\bibliographystyle{plain}
\bibliography{dahema}

\end{document}